\documentclass[10pt]{amsart}

\usepackage{xcolor}
\definecolor{winered}{rgb}{0.8,0,0}
\definecolor{deepblue}{rgb}{0,0,0.8}
\usepackage[colorlinks=true]{hyperref}
\hypersetup{linkcolor=winered, citecolor=deepblue}
\usepackage{amsmath}
\usepackage{amsthm}
\usepackage{amssymb}
\usepackage{mathrsfs}
\usepackage{mathtools}
\usepackage{tikz-cd}
\usepackage[color=cyan]{todonotes}
\usepackage{comment}
\usepackage{thmtools}   % Theorem-like environments, extends amsthm
\usepackage{mathtools}  % Fonts and environments for mathematical formulae
\usepackage[shortlabels]{enumitem}

\DeclareRobustCommand{\SkipTocEntry}[5]{}

\declaretheoremstyle[headfont   = \bfseries\sffamily,
                     notefont   = \normalfont,
                     bodyfont   = \itshape,
                     spaceabove = 6pt,
                     spacebelow = 6pt]{plain}
\declaretheoremstyle[headfont   = \bfseries\sffamily,
                     notefont   = \normalfont,
                     spaceabove = 6pt,
                     spacebelow = 6pt]{definition}
\declaretheorem[style = plain,  numberwithin = section]{theorem}
\declaretheorem[style = plain,       sibling = theorem]{corollary}
\declaretheorem[style = plain,       sibling = theorem]{lemma}
\declaretheorem[style = plain,       sibling = theorem]{proposition}
\declaretheorem[style = plain,       sibling = theorem]{conjecture}

\declaretheorem[style = definition,  sibling = theorem]{definition}
\declaretheorem[style = definition,  sibling = theorem]{example}
\declaretheorem[style = definition,  sibling = theorem]{properties}
\declaretheorem[style = remark,      sibling = theorem]{remark}

\newcommand{\A}{\mathbb{A}}

\newcommand{\G}{\mathbb{G}}

\newcommand{\N}{\mathbb{N}}
\renewcommand{\P}{\mathbb{P}}
\newcommand{\Q}{\mathbb{Q}}

\newcommand{\T}{\mathbb{T}}

\newcommand{\Z}{\mathbb{Z}}

\newcommand{\cA}{\mathcal{A}}

\newcommand{\cC}{\mathcal{C}}

\newcommand{\cF}{\mathcal{F}}
\newcommand{\cG}{\mathcal{G}}

\newcommand{\cM}{\mathcal{M}}

\newcommand{\cO}{\mathcal{O}}

\newcommand{\rD}{\mathrm{D}}

\newcommand{\rN}{\mathrm{N}}

\newcommand{\gp}{\mathrm{gp}}

\newcommand{\sat}{\mathrm{sat}}
\DeclareMathOperator{\Hom}{Hom}

\DeclareMathOperator{\Spec}{Spec}

\newcommand{\id}{\mathrm{id}}
\newcommand{\ul}{\underline}
\newcommand{\ol}{\overline}

\newcommand{\fib}{\mathrm{fib}}

\newcommand{\ket}{\mathrm{k\acute{e}t}}
\newcommand{\set}{\mathrm{s\acute{e}t}}
\newcommand{\et}{\mathrm{\acute{e}t}}
\newcommand{\h}{\mathrm{h}}
\newcommand{\lh}{\mathrm{lh}}

\newcommand{\letale}{\mathrm{l\acute{e}t}}

\newcommand{\SmlSm}{\mathrm{SmlSm}}
\newcommand{\Sh}{\mathrm{Sh}}

\newcommand{\Sm}{\mathrm{Sm}}
\newcommand{\lSm}{\mathrm{lSm}}

\newcommand{\Sp}{\mathrm{Sp}}
\newcommand{\lSch}{\mathrm{lSch}}
\DeclareMathOperator{\rightrightrightarrows}{\substack{\rightarrow \\[-1em] \rightarrow \\[-1em] \rightarrow}}
\DeclareMathOperator{\rightrightrightrightarrows}{\substack{\rightarrow\\[-1em] \rightarrow \\[-1em] \rightarrow\\[-1em] \rightarrow}}
\newcommand{\eff}{\mathrm{eff}}

\newcommand{\pt}{\mathrm{pt}}
\newcommand{\red}{\mathrm{red}}

\newcommand{\logSH}{\mathrm{logSH}}
\newcommand{\SH}{\mathrm{SH}}

\usepackage{varioref}
\usepackage{hyperref}
\usepackage[capitalize, nameinlink]{cleveref}

\crefname{conjecture}{Conjecture}{Conjectures}

\begin{document}

\author{Nikolai Opdan}
\address{Department of Mathematics, University of Oslo, 
Niels Henrik Abels hus, Moltke Moes vei 35, 0851 Oslo, Norway}
\email{ntmarti@math.uio.no}

\author{Doosung Park}
\address{Bergische Universit{\"a}t Wuppertal,
Fakult{\"a}t Mathematik und Naturwissenschaften
\\
Gau{\ss}strasse 20, 42119 Wuppertal, Germany}
\email{dpark@uni-wuppertal.de}

\author{Paul Arne {\O}stv{\ae}r}
\address{Dipartimento di Matematica ``Federigo Enriques'', 
Universit\`a degli Studi di Milano, Via Cesare Saldini 50, 20133 Milano, Italy} 
\email{paul.oestvaer@unimi.it}
\address{Department of Mathematics, University of Oslo, 
Niels Henrik Abels hus, Moltke Moes vei 35, 0851 Oslo, Norway}
\email{paularne@math.uio.no}

\title{The logarithmic $h$- and $v$-topologies}
\begin{abstract}
We introduce the $h$- and $v$-topologies in the context of logarithmic geometry 
and discuss their applications to log étale cohomology, 
log differential forms, 
and log motives.
\end{abstract}
\maketitle

\tableofcontents

\section{Introduction}

The task at hand is to investigate logarithmic versions of the $h$- and $v$-topologies. Our perspective is drawn from the theory of logarithmic motives, as developed in recent works \cite{logDMcras}, \cite{logDM}, and \cite{logSH}. A primary motivation for this study is to analyze log differentials from the viewpoint of the log $h$-topology, similar to the approach taken for algebraic varieties in \cite{ERTL20195285}, \cite{HuberJorder}, and \cite{Huber}.

Voevodsky introduced the $h$-topology in his seminal work on motives during the 1990s \cite{Voevodsky-Homology}. One advantage of the $h$-topology is that any scheme becomes locally smooth, as established by Hironaka's theorem on the resolution of singularities in characteristic zero. The concept of extending invariants on smooth schemes through $h$-hypercoverings dates back to Deligne's definition of mixed Hodge structures for proper varieties \cite{HodgeIII}. This approach has led to diverse applications that highlight the intricate interplay between algebra, geometry, and topology, notably in Suslin-Voevodsky's work on the homology of algebraic varieties \cite{SV}. More recently, in the realm of perfectoid geometry, Bhatt and Scholze \cite{Bhatt-Scholze} introduced the $v$-topology, also known as the universally subtrusive topology. This topology is grounded in the work of Rydh \cite{Rydh}. The $v$-topology is essentially generated by Zariski coverings and quasi-compact morphisms that fulfill a lifting property for specializations. It is finer than the $h$-topology and provides powerful techniques for establishing results in arithmetic geometry. Notably, a general $v$-cover can be viewed as a limit of $h$-covers. The $h$- and $v$-topologies coincide for finite-type morphisms of Noetherian schemes, meaning that étale covers and proper surjections generate these topologies.

Bhatt and Scholze referred to the notion of a $v$-cover to indicate surjectivity concerning valuations on adic spaces \cite[Remark 2.2]{Bhatt-Scholze}. To adapt this concept for log schemes, we begin by discussing valuative monoids in \Cref{section:Valuative monoids}. For our purposes, the critical notion is that of a log valuation ring $(V,P)$, which consists of a valuation ring $V$ and a valuative commutative monoid $P$ (see \cref{df:lvr}). \cref{thm:logvaluationringsliftspecializations} demonstrates that if $X$ is an integral log scheme and $x' \rightsquigarrow x$ represents a specialization of points in $X$, then there exists a log valuation ring $(V,P)$ and a morphism $\Spec{(V,P)} \to X$ such that the specialization $\mathfrak{m} \rightsquigarrow (0)$ lifts $x' \rightsquigarrow x$ for the maximal ideal $\mathfrak{m} \subset V$. The property of lifting specializations of points on log schemes makes the notion of log valuation rings a valuable generalization of valuation rings; for further background, see \cite[Tag 01J8]{stacks-project}.
\vspace{0,1in}

\Cref{section:Log v-covers} introduces and explores fundamental properties of our log $h$- and $v$-topologies within the framework of fine and saturated (fs) log schemes. If $P$ is a valuative monoid and $f: \Delta \to \Sigma$ is a subdivision of fans, \cref{thm:subdivisionsoffans} states that any morphism of monoschemes $g: \Spec{P} \to \Sigma$ induces an isomorphism
\begin{equation}
\label{equation:subdivision}
\Spec{P} \times_{\Sigma} \Delta \overset{\cong}{\to} \Spec{P}
\end{equation}
When $\Spec{P} \cong \Spec{\mathbb{N}}$, the morphism in \eqref{equation:subdivision} corresponds to a subdivision of $\mathbb{N}$ by \cite[Lemma A.3.11]{logDM}, and thus it is indeed an isomorphism. In general, the proof relies on the idea that the dual fan of any valuative monoid is too rigid to accommodate exotic subdivisions. We leverage this reasoning to conclude that \eqref{equation:subdivision} is an isomorphism. By employing the technique of subdivision, we can argue locally. This method enables us to identify quasi-compact log $v$-covers with universally subtrusive morphisms of fs log schemes; see \cref{thm:mainsubtrusive} for a precise formulation. The section concludes with a log version of Zariski's structure theorem, which is of independent interest; see \cref{thm:structuretheorem}.
\vspace{0,1in}

In \cref{section:Examples of log h-sheaves}, we turn our attention to applications of Kummer étale cohomology, log differential forms, and log motives. Our result on log \( v \)-descent for log étale cohomology in \Cref{subsection:Kummer étale cohomology with torsion coefficients} implies that if \( S \) is an fs log scheme equipped with the trivial log structure and \( \cG \) is a torsion sheaf on the small log étale site \( S_{\letale} \), then the presheaf on \( \lSch_{qcqs}/S \) — the category of quasi-compact, quasi-separated fs log schemes over \( S \) — given by 
\[
(\alpha\colon X\to S) \mapsto H^i_{\letale}(X,\alpha^*\cG)
\]
is a log \( v \)-sheaf of abelian groups for all \( i \geq 0 \) (see \Cref{thm:etalecohomology} for a more precise formulation). This descent result serves as a log-geometric analog of the fact that étale cohomology satisfies \( v \)-descent, as seen in \cite[Exposé Vbis]{SGA4} and \cite[Proposition 5.3.3]{CDEtale}. Our proof is not a straightforward generalization of the one for schemes. Through various reduction steps, we show that on the Kummer étale site, every log \( v \)-cover \( f\colon Y \to X \) of fs log schemes satisfies universal descent for the finitary presheaf given by 

\[
(\alpha \colon X \to S) \mapsto R\Gamma_{\ket}(X,\alpha^*\cG)
\]

To achieve this, we employ a localization argument in log étale cohomology for the strict closed immersion \( i\colon\pt\rightarrow\square:= (\mathbb{P}^1,\infty) \) from a point to the logarithmic unit interval and its strict open complement \( j\colon(\mathbb{P}^1-\{0\},\infty)\rightarrow\square \). Another crucial step involves reducing the proof of log \( v \)-descent to strict morphisms and morphisms of fs log schemes whose underlying morphism of schemes is an isomorphism. We also use the observation that dividing covers of fs log schemes, as defined in \cite{logDM}, are special cases of log \( v \)-covers.

Inspired by \cite{HuberJorder}, we study the log \( h \)-sheafification of the sheaf of differential forms over fields of characteristic \( 0 \) in \Cref{subsection:Logarithmic differentials in characteristic 0}. Building on 
\cref{conj:loghdifferantialshomotopyinvariant} and \cref{conj:loghdifferentialsequal}, which we hope to address in future work, we prove in \cref{thm:loghdifferentials} that log \( h \)-sheafification preserves differentials defined on log smooth log schemes. The key component of our proof is the logarithmic Gysin sequence constructed in \cite{logDM}, which facilitates an inductive argument based on the number of irreducible components in the boundary of the log scheme. As shown in \cref{cor:loghsheafificationisaloghsheaf}, this implies that the sheaf of differential forms is indeed a log \( h \)-sheaf on log smooth log schemes over \( k \).

In \Cref{subsection:Logarithmic H-Motives}, we utilize the constructions from \cite{logDMcras}, \cite{logDM}, and \cite{logSH} to define symmetric monoidal stable \( \infty \)-categories of log motives for quasi-compact and quasi-separated (qcqs) base schemes \( S \) with respect to the log \( h \)-topology. As noted in \cref{rmk:loghtopologynofibreproducts}, the log \( h \)-topology does not preserve fiber products of log smooth log schemes. Therefore, to define a category of motives, we consider \( h \)-motives over fine and saturated log schemes, which need not be log smooth. Consequently, the category of logarithmic \( h \)-motives consists of log \( h \)-sheaves on \( \lSch/S \) and involves formally inverting the logarithmic unit interval \( (\mathbb{P}^1, \infty) \). The category of stable logarithmic \( h \)-motives is defined by taking \( \mathbb{P}^1 \)-spectra in this \( \infty \)-category, that is, the localization given by 

\[
\underline{\logSH}_{\lh}(S) := \Sp_{\mathbb{P}^1}(\Sh_{\lh}(\lSch/S_{qcqs},\Sp)[(\mathbb{P}^\bullet,\mathbb{P}^{\bullet-1})^{-1}])
\]

As a consequence of \cref{thm:etalecohomology}, the torsion sheaf \( \mathbb{Z}/n \) is representable in the category of \( S^1 \)-stable logarithmic \( h \)-motives for the simplicial sphere \( S^1 \) and for every integer \( n > 1 \), as demonstrated in \cref{prop:Z/nrepresentable}.
The end of \Cref{subsection:Logarithmic H-Motives} includes some speculations about triangulated categories of log \( h \)- and \( qfh \)-motives over log points.

\addtocontents{toc}{\SkipTocEntry}
\subsection*{Convention}
Our main reference for log geometry is Ogus' monograph \cite{Ogu}. Unless otherwise stated, every log scheme in this paper has a Zariski log structure.

\addtocontents{toc}{\SkipTocEntry}
\subsection*{Acknowledgments}
Our work is supported by the RCN Project no. 312472, titled "Equations in Motivic Homotopy," and The European Commission — Horizon-MSCA-PF-2022 "Motivic integral \( p \)-adic cohomologies." D.P. is supported by the research training group GRK 2240 "Algebro-Geometric Methods in Algebra, Arithmetic and Topology."
We thank to the referee for a detailed report that helped to streamline the exposition.

\section{Log valuation rings}
\label{section:Valuative monoids}
This section introduces the notion of log valuation rings. 
Our main result, see \cref{thm:logvaluationringsliftspecializations}, 
shows that log valuation rings lift specialization of points on log schemes, 
similar to how valuation rings lift specializations of points on schemes, 
see \cite[Tag 01J8]{stacks-project}.

First, 
we recall the notion of valuative monoids, 
an analog of valuation rings for monoids. 
A monoid $M$ is a set with an associative, distributive binary operation with a unit. We assume that our monoids are commutative.
Recall that if $A$ is a valuation ring with maximal ideal $m$ and fraction field $K$, then $x \in K$ implies that either $x \in A$ or $-x \in A$.

\begin{definition}
\cite[§I.2.4]{Ogu}
A \emph{valuative monoid} is an integral monoid \(P\) such that for any \(a \in P^{\gp},\) we have either \(a \in P\) or \(-a \in P.\) 
A \textit{discrete valuative monoid} is a valuative monoid such that $\overline{P} := P/P^* \cong \N.$
\end{definition}

\begin{example}
\begin{enumerate}
\item[(i)] The monoid $\N$ is a discrete valuative monoid.
\item[(ii)] Every abelian group
is a valuative monoid. 
For any ordered group \(G,\) the set of non-negative (resp.\ non-positive) elements 
\(G_{\geq 0}\) (resp. \(G_{\leq 0}\)) is a valuative monoid.
\item[(iii)] Let $A$ be a valuation ring. 
Its set of nonzero elements $A-\{0\}$ is a valuative monoid with 
$(A-\{0\})^{\gp} = K^\times$, where $K$ is the field of fractions of $A$.
\end{enumerate}
\end{example}

\begin{properties}
The following facts are straightforward generalizations of properties concerning valuation rings; see \cite[§I.2.4]{Ogu}.
\begin{enumerate}
\item[(i)] For any integral monoid \(M\), there exists a valuative monoid \(P \subset M^{\gp}\) containing \(M\) such that \(M^* = M \cap P^*,\) where \((-)^*\) denotes the subgroup of \((-)\) consisting of all invertible elements.
\item[(ii)] A monoid \(P\) is valuative if and only if it is a maximal element in the set of submonoids of \(P^{\gp}\) for the partial ordering
\[Q \leq Q' \Leftrightarrow \big( Q \subset Q' \text{ and } Q^* = Q\cap (Q')^* \big).\]
\item[(iii)] The saturation \(P^{\sat}\) of a monoid \(P\) coincides with the intersection of all valuative monoids \(V \subset P^{\gp}\) containing \(P\) such that \(P^* = P \cap V^*.\) This implies that all valuative monoids are saturated; an analog of the fact that a valuation ring $A$ is maximal with respect to the relation of being dominant among all local rings contained in the fraction field of $A$.
\item[(iv)] 
If $H \subset G$ is an inclusion of abelian groups, and $P \subset G$ is a valuative monoid, then $P \cap H$ is a valuative monoid.
\item[(v)] If $P$ is a valuative monoid and $\mathfrak{p}$ is an ideal, then the localization $P_\mathfrak{p}$ is valuative.
\end{enumerate}
\end{properties}

Recall that a monoid $M$ is integral if the inclusion $M \to M^{\gp}$ is injective. 
One technical difficulty in log geometry is that the tensor product of integral monoids 
may not be integral. 
We consider homomorphisms of integral monoids $P \to Q$ having the property that for any other 
homomorphism of integral monoids $P \to P'$, 
the pushout $Q \oplus_P P' = Q \otimes_P P'$ is integral.
Such homomorphisms are called \emph{integral}.

\begin{lemma}\label{lem:ValuativeIntegral}
If $P$ is valuative, 
then every homomorphism of integral monoids $P \to Q$ is integral.
\end{lemma}
\begin{proof}
See \cite[Proposition I.4.6.3 (5)]{Ogu}.
\end{proof}

Recall from \cite[Definition I.4.3.1]{Ogu} that a homomorphism of monoids $\phi \colon P \to Q$ is called \textit{Kummer} if it is injective and \(\phi \otimes \Q \colon P \otimes \Q \to Q \otimes \Q\) is surjective, i.e., if it is injective and for all $q \in Q$ there exists $n \in \N$ and $p \in P$ such that $q = n\phi(p)$.

Recall that a homomorphism of integral monoids $\phi \colon P \to Q$ is:
\begin{enumerate}
\item[(i)] \emph{Logarithmic} if the induced homomorphism $\phi^{-1}(Q^*) \to Q^*$ is an isomorphism \cite[Definition I.4.1.1(3)]{Ogu}.
\item[(ii)] \emph{Flat} if for every functor $F$ from a finite connected category to the category of $P$-sets, the natural homomorphism
\[(\lim F) \otimes_P Q \to \lim(F \otimes_P Q)\]
is an isomorphism \cite[Definition I.4.5.2]{Ogu}.
\item[(iii)] \emph{Local} if $\phi^{-1}(Q^*) = P^*$ \cite[Definition I.4.1.1 (1)]{Ogu}.
\item[(iv)] \emph{Exact} if the induced square
\begin{center}
\begin{tikzcd}
P \arrow[r] \arrow[d] & Q \arrow[d] \\
P^{\gp} \arrow[r] & Q^{\gp}
\end{tikzcd}
\end{center}
is cartesian \cite[Definition I.2.1.15]{Ogu}, i.e., 
if there is an isomorphism $ P \cong P^{\gp} \times_{Q^{\gp}} Q.$
\item[(vi)] \textit{Locally exact} if for every face $G$ of $Q$, the localized homomorphism of monoids $P_{\phi^{-1}(G)} \to Q_G$ is exact \cite[Definition I.4.2.12]{Ogu}.
\end{enumerate}

An exact homomorphism is local \cite[Proposition I.2.1.16 (2)]{Ogu}, and if $P$ is valuative and $f$ local, then it is exact \cite[Proposition I.4.2.1 (4)]{Ogu}. A locally exact morphism is locally surjective. 

\begin{lemma}\label{lem:locallyexact}
A homomorphism $\theta \colon P \to Q$ of monoids is locally exact if one of the following conditions is satisfied:
\begin{enumerate}
\item[(i)] $\theta$ is flat.
\item[(ii)] $P$ is valuative.
\item[(iii)] $Q$ is a group.
\end{enumerate}
\end{lemma}
\begin{proof}
If $\theta$ is integral, then it is locally exact by \cite[Proposition I.4.6.3 (4)]{Ogu}. Then (i) follows from \cite[Proposition I.4.6.7]{Ogu}, and (ii) follows from \cref{lem:ValuativeIntegral} and (iii) follows from \cite[Proposition I.4.6.3 (5)]{Ogu}. 
\end{proof}

The ``going down'' theorem for rings \cite[Tag 00HS]{stacks-project} takes the following form for monoids.
 
\begin{proposition}[Going down for monoids]\label{prop:goingdown}
Let $\theta \colon P \to Q$ be a locally exact homomorphism of integral monoids. Let $\mathfrak{p}' \subset \mathfrak{p}$ be a chain of prime ideals in $P$ and $\mathfrak{q} \subset Q$ a prime ideal such that $\theta^{-1}(\mathfrak{q}) = \mathfrak{p}.$ Then there exists a prime ideal $\mathfrak{q}' \subset Q$ contained in $\mathfrak{q}$ which lifts $\mathfrak{p}'$.
\end{proposition}
\begin{proof}
It follows from \cite[Proposition I.4.2.2]{Ogu} that a locally exact homomorphism of integral monoids is locally surjective, hence for every point $x \in \Spec{Q}$, corresponding to a prime ideal $\mathfrak{q} \subset Q$, the map of topological spaces $\Spec{\theta}_{x} \colon \Spec{Q_{\mathfrak{q}}} \to \Spec{P_\mathfrak{p}}$ is surjective. Therefore, if $\mathfrak{p}' \subset \mathfrak{p}$ is a chain of prime ideals in $P$, there exists a prime $\mathfrak{q}'$ of $Q_\mathfrak{q}$ lifting the prime ideal $\mathfrak{p}'+P_\mathfrak{p}$ of $P_\mathfrak{p}.$ The inverse image of $\mathfrak{q}'$ under the localization $Q \to Q_{\mathfrak{q}}$ is the prime ideal we are looking for.
\end{proof}

Working with valuative monoids we have a simpler description of faithfully flat homomorphisms. 

\begin{proposition}[Characterization of faithfully flat homomorphisms] \label{prop:Faithfullyflatcharacterization}
Let $\theta\colon P\to Q$ be a homomorphism of integral monoids.
Assume that $P$ is valuative and $P^\gp$ is torsion free.
Then $\theta$ is
injective and local if and only if it is faithfully flat, i.e., it is flat and $\Spec{\theta}$ is surjective.
\end{proposition}
\begin{proof}
If $\theta$ is injective and local,
then $\theta$ is integral and exact by \cite[Proposition I.4.6.3 (4),(5)]{Ogu}.
Hence $\theta$ is flat by \cite[Proposition I.4.6.8]{Ogu},
and $\Spec(\theta)$ is surjective by \cite[Proposition I.4.2.2]{Ogu}.

Conversely, if $\theta$ is flat and $\Spec(\theta)$ is surjective, 
the homomorphism of monoid algebras $\Z[\theta] \colon \Z[P] \to \Z[Q]$ 
is a flat ring homomorphism \cite[Proposition I.4.5.12]{Ogu}. 
Thus, 
$\Z[\theta]$ is injective since $\Z[P]$ is an integral domain and $\Z[Q]\neq 0$.
Hence $\theta$ is injective.
Since $\Spec(\theta)$ is surjective,
there exists a prime ideal $\mathfrak{q}$ of $Q$ such that $\theta^{-1}(\mathfrak{q})=P^+$.
This implies $\theta^{-1}(Q^+)=P^+$,
i.e., $\theta$ is local.
\end{proof}

Next we introduce the notion of log valuation rings. 
Recall that a \emph{log ring} is a pair $(A, M)$ 
with a homomorphism of monoids $\beta \colon M \to (A,\times)$, 
where $(A,\times)$ is the underlying multiplicative monoid of $A$, 
see \cite[Definition III.1.2.3]{Ogu}.

\begin{definition}
\label{df:lvr}
\begin{enumerate}
\item A \emph{log valuation ring} is a log ring $(V,P)$ such that $V$ is a valuation ring, 
$P$ is a valuative monoid, and $P\to V$ is logarithmic.
\item A \textit{discrete} log valuation ring is a log valuation ring $(V,P)$ such that $V$ 
is a discrete valuation ring and $P$ is a discrete valuative monoid.
\end{enumerate}
\end{definition}

\begin{definition}
For a monoid $P$,
let $\A_{P}$ denote the log scheme associated to the log ring $P \to \Z[P].$ 
For a sharp fs monoid $Q$, 
we define $\pt_{Q}$ to be the fs log scheme $$\pt_{Q}:=\A_Q\times_{\ul{\A_Q}}\{O\}$$ 
where $O$ denotes the origin of $\A_Q$.
For a ring $R$, 
we set $$\A_{P,R}:=\Spec{R}\times \A_P$$ and $\pt_{Q,R}:=\Spec{R}\times \pt_Q$,
where $\times$ means the fiber product over $\Spec{\Z}$.
\end{definition}

\begin{example}
\label{exm:log point}
Let $k$ be a field.
Then every fs log scheme $X$ whose underlying scheme is $\Spec{k}$ is isomorphic to $\pt_{P,k}$ for some sharp fs monoid $P$.
\end{example}

\begin{example}
The log ring $(k,k^*\oplus \N)$ for a field $k$ as in \cite[Example III.1.5.2]{Ogu} 
is a log valuation ring whose spectrum is the standard log point $\pt_\N$.
\end{example}

Recall from \cite[Tag 0ASG]{stacks-project} that a homomorphism $V \to W$ of valuation rings is called an \emph{extension of valuation rings} if it is injective and local. 

\begin{definition}
An \emph{extension of log valuation rings} $\theta \colon (V,P) \to (W,Q)$ is a 
homomorphism of log rings $(V,P) \to (W,Q)$ such that $V \to W$ is an
extension of valuation rings.
\end{definition}

Every valuation ring \(V\) gives a log valuation ring \((V, V-\{0\}).\)
If \((A,M)\) is a log ring, we define a log scheme \(\Spec(A, M)\) whose underlying scheme is 
$\Spec{A}$ and log structure is associated with the prelog structure \(M \to \mathcal{O}_{\Spec{A}}\) 
induced by the homomorphism $M \to A$ \cite[Definition III.1.2.3]{Ogu}. 

\begin{proposition}
The functor
\[\{
\mathrm{Log}
\text{ }
\mathrm{valuation}
\text{ }
\mathrm{rings}\}^\mathrm{op}
\longrightarrow
\{
\mathrm{Log}
\text{ }
\mathrm{schemes}
\}\]
given by $(V,P)\mapsto \Spec{(V,P)}$ is fully faithful.
\end{proposition}
\begin{proof}
For a log valuation ring $(V,P)$,
observe that the induced homomorphism
\[
P\to \Gamma(\Spec{(V,P)},\cM_{\Spec{(V,P)}})
\]
is an isomorphism.
From this point on, one can argue as in \cite[Proposition 7.7]{BLPO}.
\end{proof}

\begin{example}\cite[Example III.1.5.6]{Ogu}
\label{exm:logvaluationring}
If $V$ is a discrete valuation ring, then the log scheme corresponding to the log ring $(V, V-\{0\})$, where $V-\{0\}$ is seen as the multiplicative monoid of $V,$ is called a \emph{standard log dash}, and is an example of a log valuation ring.
The sheaf of monoids $\mathcal{M}_{(\Spec{V,V-\{0\}})}$ can be described explicitly as follows: The underlying scheme $\Spec{V}$ has only two points, 
the unique closed point $\mathfrak{m}$ and $(0)$ the unique open point, 
so there is a commutative diagram
\begin{center}
 \begin{tikzcd}
 V-\{0\} \ar[d] \ar[r, "="] & \mathcal{M}_{V,\mathfrak{m}} \ar[r] \ar[d] & \mathcal{O}_{V, \mathfrak{m}} = V \ar[d] \\
 K^\ast \ar[r, "="] & \mathcal{M}_{V, (0)} \ar[r] & \mathcal{O}_{V,(0)} = K
 \end{tikzcd}
\end{center}
where $K$ denotes the fraction field of $V$. 
The log structure of $(V, V-\{0\})$ may also be described as the compactifying log structure 
for the inclusion of the open immersion $\Spec K \to \Spec V$.
\end{example}

Recall that if $x,x' \in X$, then $x$ is a \emph{specialization} of $x'$
(or $x'$ is a \emph{generalization} of $x$) if $x \in \ol{\{x'\}}.$ 
In this case, we write $x \leq x'$. 
A continuous map of topological spaces
$f \colon Y \to X$ \emph{lifts specializations} if for every specialization $x \leq x'$ in $X$ and point $y' \in Y$ such that $f(y') = x'$ there is a specialization $y\leq y'$ in $Y$ such that $f(y) = x.$ Dually, a continuous morphism $f \colon Y \to X$ \emph{lifts generalizations} if for every generalization $x \leq x'$ and a point $y$ in $Y$ there is a generalization $y \leq y'$ such that $f(y') = x'.$ A map of topological spaces is \emph{stable under generalizations} (resp. specializations) if it preserves 
generalizations (resp. specializations).

Valuation rings lift specializations for schemes \cite[Tag 01J8]{stacks-project}. 
The following theorem shows that log valuation rings play the same role in log geometry.

\begin{theorem}\label{thm:logvaluationringsliftspecializations}
Let $X$ be a quasi-coherent integral log scheme, 
i.e., $X$ admits integral charts Zariski locally.
Suppose se that $x' \rightsquigarrow x$ is a specialization
of points in $X$. 
Then there exists a log valuation ring $(V,P)$, 
with maximal ideal $\mathfrak{m} \subset V$, 
and a morphism $\Spec{(V,P)} \to X$ such that the specialization
$\mathfrak{m} \rightsquigarrow (0)$ lifts $x' \rightsquigarrow x$.
\end{theorem}
\begin{proof}
We may assume $X$ is the spectrum of a log ring $(A,M),$ 
where $A$ is an integral local ring and $M$ is a monoid, 
and where the generic (resp.\ unique closed) point of $X$ is $x'$ (resp.\ $x$).
Let $K$ be the fraction field of $A$.
A valuation ring $V$ exists with fraction field $K$ dominating $A$.
We need to show that the homomorphism of monoids $\theta\colon M\to V$ 
factors through a valuative monoid $P$.
It will be convenient to use additive instead of multiplicative notation for the 
monoid operation in a ring.
To avoid confusion,
we write $Q$ for the monoid $V$ in additive notation
and let $*$ be the element of $Q$ corresponding to $0\in V$.
Observe that $*+q=*$ for every $q\in Q$.

Let $\cA$ be the set of extensions $N\to Q$ of $\theta$ such that $N$ is an integral monoid,
$N$ dominates $M$,
and $M^\gp\to N^\gp$ is an isomorphism.
If $\{N_\alpha\to Q\}$ is a chain in $\cA$,
then $\bigcup_\alpha N_\alpha \to Q$ is an element of $\cA$.
Hence, 
by the Hausdorff maximality principle, 
a maximal element $\eta\colon P\to Q$ of $\cA$ exists.

We claim that $P$ is valuative.
If not, an element $x\in P^\gp$ exists such as $x,-x\notin P$.
Consider the face $F:=\eta^{-1}(Q-\{*\})$ of $P$,
and let $\alpha\colon F\to Q-\{*\}$ be the restriction of $\eta$.
First, assume $x\in F^\gp$.
Since $Q-\{*\}$ is valuative,
at least one of $\alpha^\gp(x)$ and $\alpha^\gp(-x)$ is in $Q-\{*\}$,
and we may assume $q:=\alpha^\gp(x)\in Q-\{*\}$ without loss of generality.
Let $P'$ be the submonoid of $P^{\gp}$ generated by $P$ and $x$.
We claim that $\eta'\colon P'\to Q$ given by $\eta'(p+nx):=\eta(p)+n*$ for $p\in P$ and $n\in \N$ is well-defined.
Assume that we have $p+nx=p'+n'x$ with $p'\in P$ and $n'\in \N$.
If $n,n'>0$,
then $\eta'(p+nx)=*=\eta'(p'+n'x)$.
In all other cases, we may assume $n'=0$ without loss of generality.
If $p\notin F$, then $p\notin F^\gp$,
so $p'\notin F$ since $x\in F^\gp$.
Hence $\eta'(p+nx)=\eta'(p')=*$.
If $p\in F$,
then $p'\in F$ since $x\in F^\gp$.
Thus, 
$\eta'(p+nx)=\eta'(p')=\alpha^\gp(p+nx)$, 
so that $\eta'$ is well-defined.
Hence $\eta'$ is an extension of $\eta$,
which contradicts the maximality of $\eta$.

Next, we assume $x\notin F^\gp$.
If $p+nx,p'-n'x\in F$ for some $p,p'\in P$ and $n,n'\in \N^+$,
then $n'p+np'\in F$,
so $p\in F$.
This implies $n(x+p)\in F$.
Since $F$ is saturated,
we have $x+p\in F$ and hence $x\in F^\gp$,
which is a contradiction.
Hence we may assume that $p+nx\notin F$ for every $p\in P$ and $n\in \N^+$.
Let $P'$ be the submonoid of $P^\gp$ generated by $P$ and $x$.
Next we claim that $\eta'\colon P'\to Q$ given by $\eta'(p+nx):=\eta(p)+n\ast$ is well-defined.
Assume that $p+nx=p'+n'x$, where $p'\in P$ and $n\in \N$.
If $n,n'>0$,
then $\eta'(p+nx)=*=\eta'(p'+n'x)$.
In all other cases, we assume $n'=0$ without loss of generality.
If $n=0$,
then we have $\eta'(p+nx)=\eta'(p')$.
If $n>0$,
then $\eta'(p+nx)=*$.
On the other hand,
the assumption on $x$ implies $p'\notin F$.
Thus, $\eta(p')=*$, so that $\eta'$ is well-defined.
Hence $\eta'$ extends $\eta$,
contradicting the maximality of $\eta$.
\end{proof}

\section{Log \texorpdfstring{$v$}{v}-covers for fs log schemes}
\label{section:Log v-covers}
In this section, 
we introduce the covers in the log $h$- and $v$-topologies.

Our definition of subtrusive morphisms for log schemes is based on the same notion for schemes,
which requires the constructible topology.
Recall that a 
subset of a topological space $X$ is constructible if it is the finite union of subsets of the form 
$U \cap V^{c},$ where $U, V\subset X$ are open and retro-compact \cite[Tag 005G]{stacks-project}. 
A subset is pro-constructible if locally it is the intersection of constructible subsets.
The constructible topology on schemes is the topology whose closed subsets are the pro-constructible 
subsets \cite[§7.2]{EGA}.
A morphism of schemes $X \to Y$ is called
\begin{enumerate}
\item[(i)] \emph{submersive} if it is surjective and the topology on $Y$ is the quotient topology 
induced by $f$ \cite[p.\ 5]{Rydh}.
\item[(ii)] \emph{subtrusive} if it is submersive in the constructible topology 
(for instance if $f$ satisfies one of the conditions of \cite[Proposition 1.6]{Rydh}) and if
every ordered pair $y\leq y'$ of points of $Y$ lifts to an ordered pair $x\leq x'$ of points of $X$.
\end{enumerate}

We recall from \cite[pp.\ 187, 190]{Rydh} some properties of submersive and subtrusive morphisms:
Every subtrusive morphism of schemes is submersive.
The composition of two submersive (resp.\ subtrusive) morphisms is submersive (resp.\ subtrusive), and if the composition $f \circ g$ is submersive (resp.\ subtrusive), 
then $g$ is submersive (resp.\ subtrusive).
Typical examples of morphisms that are not submersive are non-flat morphisms or morphisms obtained by removing points from blow-ups (to force non-properness).
We will apply the following definition to (fs log) schemes.

\begin{definition}
Let $\cC$ be a category with fiber products,
and let $\mathbf{P}$ be a property of morphisms in $\cC$.
A morphism $X \to S$ in $\cC$ has the property $\mathbf{P}$ \emph{universally} if for any other morphism 
$S' \to S$ in $\cC$ the pullback morphism $S' \times_S X \to S'$ has the property $\mathbf{P}.$
\end{definition}

We will repeatedly use that a surjective morphism of schemes is universally subtrusive 
if it is proper or universally open, see \cite[Remark 2.5]{Rydh}.
Our conventions are that 
a morphism of log schemes is \emph{quasi-compact, quasi-separated, separated, proper, 
finite type, locally finitely presented, 
finitely presented, surjective, or open} if its underlying morphism of schemes 
satisfies the same property.
The constructible topology for log schemes is the constructible topology on its 
underlying scheme.

\begin{definition}
A morphism of log schemes is
\emph{subtrusive} (resp. \emph{submersive}) 
if its underlying morphism of schemes 
is subtrusive (resp. \emph{submersive}).
\end{definition}

Next, we define the log $h$-topology for quasi-compact fs log schemes, which serves as a logarithmic analog of the $h$-topology presented in \cite[Definition 8.1]{Rydh}. Our log $h$-topology aligns with the original $h$-topology defined by Voevodsky in \cite[Definition 3.1.2]{Voevodsky-Homology} for the noetherian case.

\begin{definition}
\label{definition:loghtop}
The \emph{log $h$-topology} on the category of qcqs fs log schemes is generated by finite covering 
families $\{U_i\to X\}_{i\in I}$ such that $\amalg_{i\in I} U_i\to X$ is universally subtrusive and 
finitely presented.
We use lh as a shorthand for the log $h$-topology.
\end{definition}

\begin{remark}
The log $h$-topology is not subcanonical, i.e., not all representable presheaves are sheaves in the log $h$-topology. For instance, the structure sheaf $\mathcal{O}_X$ is not a sheaf for the $h$-topology \cite[Tag 0EV0]{stacks-project}, so it will not be a sheaf for the log $h$-topology either.
\end{remark}

\begin{proposition}\label{cor:LoghLogv}
Let $f \colon X \to S$ be a finite type morphism of fs log schemes such that $\ul{S}$ is locally noetherian.
Then $f$ is universally submersive if and only if it is universally subtrusive.
\end{proposition}
\begin{proof}
Assume first that $f$ is universally submersive.
Let $g\colon S'\to S$ be a morphism of fs log schemes.
We need to show the pullback $f'\colon X\times_S S'\to S'$ is subtrusive.
We can work Zariski locally on $S'$ and thus assume that $g$ has a chart $P\to P'$.
To show $f'$ is subtrusive, 
it suffices to show the pullback $\ul{X\times_{\A_P} \A_{P'}}\to \ul{S\times_{\A_P} \A_{P'}}$ 
is a universally subtrusive morphism of schemes.
This follows from \cite[Corollary 2.10]{Rydh} since $f$ is universally submersive and 
$\ul{S\times_{\A_P} \A_{P'}}$ is locally noetherian.
The proof of the converse implication is the same.
\end{proof}

\begin{remark}
By \cref{cor:LoghLogv} the log $h$-topology on the category of noetherian fs log schemes is generated 
by finite covering families $\{U_i\to X\}_{i\in I}$ for which $\amalg_{i\in I} U_i\to X$ is universally 
submersive.
Example 4.5 in 
\cite{Rydh} shows that the result is false if one drops the assumption that 
$S$ is locally noetherian.
\end{remark}

\begin{example}\label{example:surjectivebutnotuniversally}
Let $k$ be a field.
Consider the following pullback square of fs log schemes
\[
\begin{tikzcd}
X' \ar[r] \ar[d] & \pt_{\N} \ar[d,"\Delta"] \\
S' \ar[r] & \pt_{\N^2}
\end{tikzcd}
\]
where $\pt_\N:=\Spec{(k,\N)}$, $\Delta$ is the diagonal morphism, and $S' \to \pt_{\N^2}$ is any nontrivial log blow-up. Then the pullback morphism $X' \to S'$ is not surjective even though the underlying morphism of schemes $\ul{\Delta}$ is an isomorphism.
\end{example}

\cref{example:surjectivebutnotuniversally} demonstrates that a morphism of fs log schemes is 
not necessarily universally surjective even though its underlying morphism of schemes 
is universally surjective. 
However, 
Kato's notion of exactness turns out to be a sufficient condition for surjective morphisms 
to be universally surjective.

\begin{proposition}\label{prop:universallysurjective}
An exact surjective morphism of fs log schemes is universally surjective in the category of fs log schemes. 
\end{proposition}
\begin{proof}
See \cite[Proposition 2.2.2]{MR1457738} for a proof.
\end{proof}

\begin{proposition}\label{prop:saturated}
Let $f\colon X\to S$ be a saturated morphism of fs log schemes.
Then $f$ is universally subtrusive in the category of fs log schemes if and only if 
$\ul{f}$ is universally subtrusive in the category of schemes.
\end{proposition}
\begin{proof}
This is an immediate consequence of the fact that for every morphism of fs log schemes $S'\to S$,
there is a natural isomorphism $\ul{X\times_S S'}\cong \ul{X}\times_{\ul{S}}\ul{S'}$.
\end{proof}

\begin{lemma}
\label{v.1}
Let $X$ be a fine log scheme with saturation $X^\sat$.
The naturally induced morphism $X^\sat\to X$ is surjective.
\end{lemma}
\begin{proof}
The question is Zariski local on $X$,
so we may assume that $X$ has a fine chart $P$.
The claim follows from the fact that the induced morphism $\ul{\A_{P^\sat}}\to \ul{\A_P}$ is universally surjective.
\end{proof}

We will frequently use the following result, 
which relates universal subtrusive morphism of log schemes with the underlying morphism of schemes.

\begin{proposition}\label{prop:exact}
\label{v.2}
Let $f\colon X\to S$ be an exact morphism of fs log schemes.
Then $\ul{f}$ is universally subtrusive if and only if 
$f$ is universally subtrusive.
\end{proposition}
\begin{proof}
If either $f$ or $\ul{f}$ is subtrusive, the other morphism is subtrusive by definition. However, the issue is to show that it is universally subtrusive because the pullback is formed in two different categories.

Suppose that $\ul{f}$ is universally subtrusive and let $S' \to S$ be any morphism of fs log schemes. We must show that the pullback morphism $g \colon X \times_S S' \to S'$ in the category of fs log schemes is subtrusive, i.e., that the morphism $\ul{g} \colon \ul{X\times_S S'} \to \ul{S'}$ is subtrusive. 
Consider the pullback $\ul{X}\times_{\ul{S}}\ul{S'} \to \ul{S'}$ of $\ul{f}$ along the underlying morphism of schemes $\ul{S'} \to \ul{S}$. Since $\ul{f}$ is universally subtrusive, this pullback morphism is universally subtrusive. 
Moreover, since $f$ is exact, the morphism of schemes
$\ul{X\times_S S'} {\longrightarrow} \ul{X}\times_{\ul{S}}\ul{S'}$ is proper and surjective by \cite[Corollary III.2.2.4]{Ogu} and Lemma \ref{v.1} and hence universally subtrusive by \cite[Remark 2.5 (1a)]{Rydh}.
Since universally subtrusive morphisms are closed under composition, 
this shows that $\ul{g}$ is universally subtrusive.

On the other hand, 
suppose that $f$ is universally subtrusive and $V \to \ul{S}$ is a morphism of schemes. 
Consider $V$ as an fs log scheme with the trivial log structure.
The projection $X\times_S (V\times_{\ul{S}}S)\to V\times_{\ul{S}} S$ is subtrusive, which implies that the projection $X\times_{\ul{S}}V\to V$ is subtrusive.
Thus $\ul{f}$ is universally subtrusive in the category of schemes.
\end{proof}

Recall from \cite[§1.10]{KatoLog2} that a morphism of fs log schemes $f\colon X\to S$ is \emph{log flat} if strict fppf locally on $S$ and $X$,
there exists a chart $P\to Q$ of $f$ such that $P\to Q$ is injective and the induced morphism $X\to S\times_{\Z[P]}\Z[Q]$ is strict flat.

A locally finitely presented flat morphism between schemes is universally open.
We have a similar property for log flat morphisms with imposing exactness.

\begin{lemma}\label{lem:logflatexact}
Let $f$ be an exact locally finitely presented log flat morphism of fs log schemes. Then $f$ is universally open.
\end{lemma}
\begin{proof}
Since the conditions on $f$ are preserved under base change, we only need to show that $f$ is open.
For a proof, 
we refer to \cite[Proposition 5.6, Remark 5.7]{quasi-section} and \cite[Proposition 2.5]{KatoLog2}.
\end{proof}

Recall from \cite[Definition 3.1.4]{logDM} and \cite{Par} that a \emph{dividing cover} 
of qcqs fs log schemes is a surjective proper log étale monomorphism.
Every log blow-up is a dividing cover by \cite[Example A.11.1]{logDM}. 
Note that a blow-up of schemes is proper and surjective, 
so it is a $v$-cover by \cite[Remark 2.5 (1a)]{Rydh}.
For fs log schemes,
every dividing cover is proper and surjective.
It is therefore reasonable to expect that log blow-ups, or more generally dividing covers, 
are examples of covers in the log $h$-topology. 
This is shown in the following proposition.

\begin{proposition}\label{prop:dividingcoversubtrusive}
Let $f$ be a dividing cover of qcqs fs log schemes. Then $f$ is universally subtrusive.
\end{proposition}
\begin{proof}
By \cite[Proposition A.11.6]{logDM} a dividing cover is universally surjective. 
To conclude, observe that proper morphisms are closed under pullbacks.
\end{proof}

The following result generalizes \cite[Remark 2.5 (2a)]{Rydh} to log schemes.

\begin{proposition}\label{prop:logflatsubtrusive}
Let $f\colon X\to S$ be a universally surjective log flat locally finitely presented morphism of fs log schemes. Then $f$ is universally subtrusive.
\end{proposition}
\begin{proof}
It suffices to show that $f$ is subtrusive since the 
conditions on $f$ are closed under pullbacks;
see \cite[Proposition IV.4.1.2 (3)]{Ogu} for log flat morphisms.
We may work Zariski locally on $S$.
By \cite[Theorem 1.1]{integrallogblowup} we may assume there is a dividing cover 
$g\colon S' \to S$ such that the projection $f'\colon X \times_S S' \to S'$ is exact.
\cref{lem:logflatexact} implies $\ul{f'}$ is universally open, 
so that $f'$ is subtrusive.
\cref{prop:dividingcoversubtrusive} shows $g$ is subtrusive.
Thus $gf'$ and hence $f$ is subtrusive.
\end{proof}

The next result is a generalization of \cite[Proposition 2.7 (vi)]{Rydh}.
\begin{lemma}\label{lem:extendinglogvaluationring}
Let $(V,P)$ be a log valuation ring and $f\colon X \to \Spec{(V,P)}$ 
a subtrusive morphism of quasi-coherent integral log schemes.
There exists a log valuation ring $(W,Q)$ and a morphism $g \colon \Spec{(W,Q)} \to X$ 
such that the composite
\begin{equation}
\label{equation:extensionlogvaluationrings}
\Spec{(W,Q)} \to X \to \Spec{(V,P)}
\end{equation}
is the spectrum of an extension of log valuation rings $(V,P) \to (W,Q)$.
\end{lemma}
\begin{proof}
Since the morphism of log schemes $X \to \Spec{(V,P)}$ is subtrusive, 
there is a lift $x' \rightsquigarrow x$ of the specialization of points 
$\mathfrak{m} \rightsquigarrow (0)$ to $X$ \cite[Proposition 2.7]{Rydh}. 
Applying \cref{thm:logvaluationringsliftspecializations} to $X$, 
we obtain a log valuation ring $(W,Q)$ and a morphism $\Spec(W,Q)\to X$ 
such that the specialization $\mathfrak{m} \rightsquigarrow (0)$ lifts $x' \rightsquigarrow x$,
where $\mathfrak{m}$ is the maximal ideal of $W$ and $(0)$ is the zero ideal of $W$.
Hence we obtain \eqref{equation:extensionlogvaluationrings}, 
which by construction gives an extension of log valuation rings $(V,P) \to (W,Q)$.
\end{proof}

By \cite[Corollary 2.9]{Rydh}, a quasi-compact morphism of schemes $f\colon X \to S$ is universally subtrusive if and only if for every valuation ring $V$ and morphism $g\colon \Spec{V} \to S,$
there is an extension $u\colon V\to W$ of valuation rings
and a commutative diagram
$$
\begin{tikzcd}
\Spec{W} \ar[r] \ar[d,"\Spec{u}"'] & X \ar[d, "f"] \\
\Spec{V} \ar[r,"g"] & S
\end{tikzcd}
$$
We use an analog of this criterion to define log $v$-covers of log schemes.

\begin{definition}
\label{v.4}
A morphism of
qcqs fs log schemes $f\colon X\to S$ is a \emph{log $v$-cover} if for every log valuation ring $(V,P)$ 
and morphism of log schemes $g\colon \Spec(V,P)\to S$,
there is an extension of log valuation rings $u\colon (V,P)\to (W,Q)$
and a commutative diagram
\begin{equation}
\label{v.4.1}
\begin{tikzcd}
\Spec(W,Q)\ar[r]\ar[d,"\Spec{u}"']&
X\ar[d,"f"]
\\
\Spec(V,P)\ar[r,"g"]&
S
\end{tikzcd}
\end{equation}
The \emph{log $v$-topology} is the topology on the category of qcqs fs log schemes 
generated by the finite covering families $\{U_i\to X\}_{i\in I}$ such that 
$\amalg_{i\in I}U_i\to X$ is a log $v$-cover. 
We use lv as a shorthand for the log $v$-topology.
\end{definition}

\begin{remark}
Any log $v$-cover is surjective according to \cref{thm:logvaluationringsliftspecializations}.
\end{remark}

\begin{proposition}\label{prop:logvcoverbasechange}
Log $v$-covers are closed under base change.
\end{proposition}
\begin{proof}
Suppose $f \colon X \to S$ is a log $v$-cover of 
qcqs fs log schemes and $S' \to S$ is a morphism of qcqs fs log schemes. 
Consider the pullback morphism $f' \colon X' \to S'$. 
If $(V,P)$ is a log valuation ring with a morphism $\Spec{(V,P)} \to S',$ 
we can form the composite $\Spec{(V,P)} \to S' \to S$. 
Since $f$ is a log $v$-cover, by assumption, 
there is a log valuation ring $(W,Q)$ and a morphism $\Spec{(W,Q)} \to X$ such that 
$(V,P) \to (W,Q)$ is an extension of log valuation rings. 
We obtain the commutative diagram
\[
\begin{tikzcd}
\Spec{(W,Q)} \ar[rr, bend left=25] \ar[d] \ar[r, dashed]
& X' \ar[r] \ar[d,"f'"] & X \ar[d, "f"] \\
\Spec{(V,P)} \ar[r] & S' \ar[r] & S
\end{tikzcd}
\]
There is a unique induced morphism $\Spec{(W,Q)}\to X'$ rendering the diagram commutative, 
so that $f'$ is a log $v$-cover.
\end{proof}

\begin{proposition}\label{prop:logvcovercomposition}
Log $v$-covers are closed under composition.
\end{proposition}
\begin{proof}
Suppose $f \colon X \to S$ and $g \colon Y \to X$ are log $v$-covers of qcqs fs log schemes. 
Let $\Spec{(V,P)} \to S$ be a morphism for a log valuation ring $(V,P)$.
There is a commutative diagram
\[
\begin{tikzcd}
\Spec{(W',Q')}\ar[d]\ar[r,"\Spec{v}"]&
\Spec{(W,Q)}\ar[d]\ar[r,"\Spec{u}"]&
\Spec{(V,P)}\ar[d]
\\
Y\ar[r,"g"]&
X\ar[r,"f"]&
S
\end{tikzcd}
\]
for extensions of log valuation rings $(V,P)\xrightarrow{u} (W,Q)\xrightarrow{v} (W',Q')$.
To conclude,
observe that $(V,P)\to (W',Q')$ is an extension of log valuation rings.
\end{proof}

For the definition of fans we refer to \cite[\S II.1.9]{Ogu}.
The \textit{support} $|\Sigma|$ of a fan $\Sigma$ in a lattice $N$ is the union of all cones of $\Sigma.$
A \emph{morphism of fans} $\Delta\to \Sigma$ in lattices $N$ and $N'$ is a homomorphism of lattices $f \colon N \to N'$ such that for every cone $\sigma\in \Delta$ we have $f(\sigma) \subset \sigma'$ for some cone $\sigma'$ in $\Sigma$. 
A \textit{partial subdivision} is a morphism of fans such that the morphism of lattices is an isomorphism. It is called a \textit{subdivision} if $|\Delta| = |\Sigma|.$
We can regard a morphism of fans $\Delta\to \Sigma$ as a morphism of connected separated fs monoschemes, see \cite[Theorem II.1.9.3]{Ogu}. 
For an fs monoid $P$, we also regard the monoscheme $\Spec{P}$ 
as a fan with a single maximal cone dual to $P$.

For morphisms of saturated monoschemes $X,Y\to S$,
let $X\times_S Y$ denote the fiber product in the category of saturated monoschemes.
The said product exists by \cite[Proposition II.1.3.5]{Ogu} and differs from the 
fiber product in the category of monoschemes in general.

For a fan $\Sigma$,
let $\T_\Sigma$ be the fs log scheme whose underlying scheme is the toric variety over $\Z$ associated with $\Sigma$ with the compactifying log structure associated with the maximal torus.

\begin{remark}\label{exm:notafan}
In general, $\A_{P}$ is not a fan for a valuative monoid $P$
(since $\ol{P}$ need not be finitely generated). 
 \end{remark}

The following result is trivially true if $P=\N$ since then the pullback morphism $\Spec{P}\times_{\Sigma}\Delta\to \Spec{P}$ corresponds to a subdivision of $\N$ by 
\cite[Lemma A.3.11]{logDM}, which is necessarily an isomorphism. 
The next result is based on the insight that the dual fan of a valuative monoid
does not admit any subdivision other than itself.

\begin{theorem}\label{thm:subdivisionsoffans}
Let $f \colon \Delta\to \Sigma$ be a subdivision of fans, and let $P$ be a valuative monoid.
Then for every morphism of monoschemes $g \colon \Spec{P}\to \Sigma$,
the projection $\Spec{P}\times_{\Sigma}\Delta\to \Spec{P}$ is an isomorphism.
\end{theorem}
\begin{proof}
As noted above, the claim holds if $P=\N$. 
This implies that $f$ is proper in the sense of \cite[Definition II.1.6.1]{Ogu}.

Now, we work with a general $P$.
Since $\Spec{P}\to \Sigma$ factors through a cone, we may assume $\Sigma=\Spec{R}$ for some fs monoid $R$.
Let $\{\Spec{Q_1},\ldots,\Spec{Q_n}\}$ be the cones of $\Delta$.
We set $P_i:=P\oplus_R Q_i$ for $1\leq i\leq n$.
Observe that we have $P^\gp\cong P_i^\gp$.

Consider the submonoid $F_i:=\{x\in P:-x\in P_i\}$ of $P$.
If $x+y\in F_i$ and $x,y\in P$,
then $-x=(-x-y)+y\in P_i$, so $x\in F_i$.
This shows that $F_i$ is a face of $P$.
Furthermore, if $x\in P_i$ and $x\notin P$, then $-x\in P$ since $P$ is valuative, so $x\in P_i^*$.
Hence we have $P_{F_i}\cong P_i$.
Thus $\Spec{P}\times_\Sigma \Delta$ is the gluing of the open subschemes 
$\Spec{P_{F_i}}$ of $\Spec{P}$.
We proof will be finished if $P_{F_i}\cong P$ for some $i$.
The image of the composite $\Spec P^\gp\to \Spec P\to \Sigma$ lies in the generic point of $\Sigma$, so the composite lifts to $\Delta$.
Since $f$ is proper,
the valuative criterion \cite[Theorem II.1.6.3]{Ogu} yields the lifting in the diagram
\[
\begin{tikzcd}
\Spec{P^\gp}\ar[d]\ar[r]&
\Delta\ar[d]
\\
\Spec{P}\ar[ru,dashed]\ar[r]&
\Sigma
\end{tikzcd}
\]
The morphism $\Spec{P}\to \Delta$ factors through $\Spec{Q_i}$ for some $1\leq i\leq n$, 
and we obtain $P_{F_i}\cong P$.
\end{proof}

\begin{proposition}\label{prop:diviningcover is a logvcover}
A dividing cover of qcqs fs log schemes
$f \colon X \to S$ is a log $v$-cover.
\end{proposition}
\begin{proof}
Let $\Spec{(V,P)} \to S$ be a morphism of log schemes,
where $(V,P)$ is a log valuation ring.
To find \eqref{v.4.1}
we can work Zariski locally on $S$.
Hence we may assume that $S$ has a chart $M$ with $M$ sharp,
see \cite[Proposition II.2.3.7]{Ogu}.
By \cite[Proposition A.11.5]{logDM}
there exists a subdivision of fans $\Sigma\to \Spec{M}$ such that $f$ factors through the projection $S\times_{\A_M} \T_\Sigma\to X$.
It suffices to find \eqref{v.4.1} for this projection,
so we may assume $X=S\times_{\A_M} \T_\Sigma$.
Since $f$ is proper and surjective,
$\ul{f}$ is a $v$-cover.
Thus there exists an extension of valuation rings $V\to W$ and a commutative diagram
\[
\begin{tikzcd}
\Spec{W}\ar[d]\ar[r]&\ul{X}\ar[d,"\ul{f}"]
\\
\Spec{V}\ar[r]&
\ul{S}
\end{tikzcd}
\]
When combined with \cref{thm:subdivisionsoffans}
we obtain a commutative diagram
\[
\begin{tikzcd}
\Spec{(W,P)}\ar[d]\ar[r]&X\ar[d,"f"]
\\
\Spec{(V,P)}\ar[r]&
S
\end{tikzcd}
\]
To conclude, 
observe that the homomorphism of log rings $(V,P)\to (W,P)$ is an extension of log valuation rings.
\end{proof}

The next theorem generalizes \cite[Corollary 2.9]{Rydh} to fs log schemes. 
\begin{theorem}\label{thm:mainsubtrusive}
Let $f \colon X \to S$ be a morphism of qcqs fs log schemes.
Then $f$ is universally subtrusive if and only if $f$ is a log $v$-cover.
\end{theorem}
\begin{proof}
Suppose $f$ is universally subtrusive in the category of fs log schemes, 
$(V,P)$ is a log valuation ring,
and $\Spec{(V,P)} \to S$ is a morphism.
We construct an extension of log valuation rings $(V,P)\to (W,Q)$ and a commutative diagram \eqref{v.4.1}.
We work Zariski locally on $Y$ and, 
due to \cref{prop:diviningcover is a logvcover}, 
dividing locally on $S$. 
Moreover, 
by \cite[Theorem 1.1]{MR1754621},
we may assume that $f$ is integral.
The fiber product $Z := X \times_S^{\log} \Spec{(V,P)}$ in the category of log schemes 
is an integral log scheme. 
The morphism $f' \colon Z \to \Spec{(V,P)}$ is subtrusive by assumption. 
\cref{lem:extendinglogvaluationring} shows
there is a morphism $\Spec(W,Q)\to Z$ such that the induced morphism of log valuation rings 
$(V,P)\to (W,Q)$ is an extension of log valuation rings. 
Thus, 
we obtain a log $v$-cover from the commutative diagram
$$
\begin{tikzcd}
\Spec{(W,Q)} \ar[r]  \ar[rr, bend left] \ar[rd]&
Z \ar[r] \ar[d] & X \ar[d, "f"] \\
&\Spec{(V,P)} \ar[r] & S
\end{tikzcd}
$$

For the other direction,
assume $f$ is a log $v$-cover.
The class of log $v$-covers is closed under pullbacks by \cref{prop:logvcoverbasechange},
so it suffices to show that $f$ is subtrusive.
Since $f$ is quasi-compact, 
we only need to show that every specialization of points $s' \rightsquigarrow s$ of $S$ 
can be lifted to $X$,
see \cite[Remark 2.5]{Rydh}.
By \cref{thm:logvaluationringsliftspecializations} there is a log valuation ring $(V,P)$ such that the 
morphism of log schemes $\Spec{(V,P)} \to S$ lifts $s' \rightsquigarrow s.$
Since $f$ is a log $v$-cover, 
there is an extension of log valuation rings $(V,P)\to (W,Q)$ fitting into \eqref{v.4.1}.
The morphism $\Spec(W,Q)\to X$ gives the required lifting of $x' \rightsquigarrow x$ in $X$ of 
$s' \rightsquigarrow s$ in $S$. 
\end{proof}

The \emph{Kummer \'etale topology} (resp.\ \emph{log \'etale topology}) 
on the category of qcqs fs log schemes is generated by families 
$\{U_i\to X\}_{i\in I}$ of Kummer log \'etale (resp.\ log \'etale) morphisms 
for $I$ finite, such that $\amalg_{i\in I} U_i\to X$ is surjective 
(resp.\ universally surjective).
The \emph{Kummer fppf-topology} (resp.\ \emph{log fppf-topology}) 
on the category of qcqs fs log schemes is generated by finite covering families of Kummer log flat 
(resp.\ log flat) finitely presented morphisms such that $\amalg_{i\in I} U_i\to X$ 
is surjective (resp.\ universally surjective).

We abbreviate these topologies by k\'et, l\'et, kfppf, and lfppf.

\begin{proposition}
\label{prop:log flat vs log v}
The log $h$-topology on the category of qcqs fs log schemes is finer than the log fppf-topology.
\end{proposition}
\begin{proof}
This is a consequence of \cref{prop:logflatsubtrusive} and \cref{thm:mainsubtrusive}.
\end{proof}

The log $cdh$-topology in \cite[Definition 2.1]{regGysin} is defined as follows.
The \emph{proper log cdh cd-structure} is the collection of cartesian squares of qcqs fs log schemes
\begin{equation}
\label{equation:proper log cdh distinguished square}
\begin{tikzcd}
W\ar[d]\ar[r]&
Y\ar[d,"f"]
\\
Z\ar[r,"i"]&
X
\end{tikzcd}
\end{equation}
where $i$ is a strict closed immersion, $f$ is a proper morphism, 
and the naturally induced morphism 
\begin{equation}\label{eq:logcdh-equation}
f^{-1}(X-Z)\longrightarrow X-Z
\end{equation}
is an isomorphism.
The \emph{log $cdh$-topology} is the coarsest topology containing strict Nisnevich coverings and 
$Z\amalg Y\to X$ as coverings for all 
squares of the form \eqref{equation:proper log cdh distinguished square}. 
We use lcdh as a shorthand for the log $cdh$-topology.

\begin{proposition}
\label{proposition:the log h-topology is finer than the log cdh-topology}
The log $h$-topology on the category of qcqs fs log schemes is finer than the log cdh-topology.
\end{proposition}
\begin{proof}
The log $h$-topology is finer than the strict Nisnevich topology,
so it suffices to show that $(i,f)\colon Z\amalg Y\to X$ admits a refinement by a log $h$-covering 
for every square \eqref{equation:proper log cdh distinguished square}.
Working Zariski locally on $X$,
by \cite[Theorem 1.1]{integrallogblowup}
we may assume there exists a dividing cover $p\colon X'\to X$ such that its pullback 
$f'\colon Y':=Y\times_X X'\to X'$ is integral.
If $i'\colon Z':=Z\times_X X'\to X'$ is the pullback of $i$, 
then $f'^{-1}(X'-Z')\to X'-Z'$ 
is an isomorphism,
which implies $(i',f')\colon Z'\amalg Y'\to X'$ is surjective since 
\eqref{eq:logcdh-equation} is an isomorphism.
Finally,
$p$ and $(i',f')$ are log $h$-coverings by \cref{v.2} and \cref{prop:diviningcover is a logvcover}.
\end{proof}

To summarize, we have the following topologies on qcqs fs log schemes, 
where $\mu \to \tau$ indicates that $\tau$ is a \emph{finer} topology than $\mu;$ 
that is, every $\mu$-covering is a $\tau$-covering. 
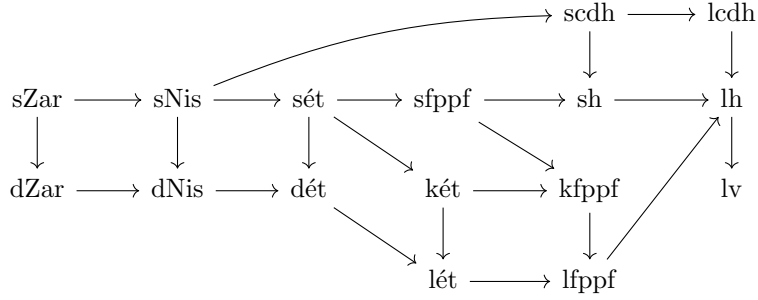
\[
\begin{tikzcd}
&
&
&
&
\mathrm{scdh}\ar[d]\ar[r]&
\mathrm{lcdh}\ar[d]
\\
\mathrm{sZar}\ar[r]\ar[d]&
\mathrm{sNis}\ar[r]\ar[d]\ar[rrru,bend left=1em]&
\set\ar[d]\ar[rd]\ar[r]&
\mathrm{sfppf}\ar[rd]\ar[r]&
\mathrm{sh}\ar[r]&
\mathrm{lh}\ar[d]
\\
\mathrm{dZar}\ar[r]&
\mathrm{dNis}\ar[r]&
\mathrm{d\acute{e}t}\ar[rd]&
\ket\ar[d]\ar[r]&
\mathrm{kfppf}\ar[d]&
\mathrm{lv}
\\
&
&
&
\letale\ar[r]&
\mathrm{lfppf}\ar[ruu]
\end{tikzcd}
\]
\vspace{0.1in}

The strict topologies 
$\mathrm{sZar}$, $\mathrm{sNis}$, $\mathrm{sfppf}$, $\mathrm{scdh}$, and $\mathrm{sh}$ are generated by families of strict morphisms $\{U_i\to X\}_{i\in I}$ such that $\{\ul{U_i}\to \ul{X}\}_{i\in I}$ is the covering for the corresponding topology of schemes.
We refer to \cite[Definition 3.1.4]{logDM} for the dividing Zariski topology (dZar), dividing Nisnevich topology (dNis), and dividing \'etale topology (d\'et).
\vspace{0.1in}

We end this section with a logarithmic analog of \cite[Theorems 3.10, 3.12]{Rydh}.

\begin{theorem}[Structure of universally subtrusive morphisms]
\label{thm:structuretheorem}
Let $f \colon X \to S$ be a universally subtrusive morphism of finite type between noetherian fs log schemes.
Then there is a refinement $g \colon Y \to S$ of $f$, and a factorization $Y\to Z \to S$ of $g$
into one of the following cases:
\begin{enumerate}
 \item[(i)] $Y\to Z$ is strict faithfully flat finitely presented, and $Z\to S$ is proper surjective.
 \item[(ii)] $Y\to Z$ is a quasi-compact Zariski covering, and $Z\to S$ is proper surjective and 
 finitely presented. 
\end{enumerate}
\end{theorem}
\begin{proof}
Working Zariski locally, we may assume $f$ has a chart. 
Since $\ul{f}$ is finitely presented, 
\cite[Theorem 1.1]{integrallogblowup} yields a dividing cover $S' \to S$ such that 
$f' \colon X\times_S S' \to S'$ is integral, 
and hence exact \cite[Theorem III.2.2.7]{Ogu}. 
Thus, by \cref{prop:dividingcoversubtrusive}, 
we may assume $f$ is exact because dividing covers are universally subtrusive, 
proper surjective, and finitely presented.
\cref{v.2} implies $\ul{f}$ is a universally subtrusive morphism.
We conclude using the corresponding results for schemes; 
see \cite[Theorem 3.10]{Rydh} for (i) and \cite[Theorem 3.12]{Rydh} for (ii).
\end{proof}

\section{Examples of log \texorpdfstring{$v$}{v}-sheaves}
\label{section:Examples of log h-sheaves}
In this section, 
we provide examples of sheaves that satisfy descent for the log $v$-topology, 
and construct a stable $\infty$-category of log $h$-motives that allows for a 
homotopy-theoretic treatment of such sheaves.

Fix a base qcqs scheme $S$ and regard it as an fs log scheme with the trivial log structure. 
Let $\lSch/S$ denote the category of fs log schemes over $S$, and let $\lSch_{qcqs}/S$ denote the full subcategory of $\lSch/S$ consisting of qcqs fs log schemes over $S$.
If $t$ is a Grothendieck topology on $\lSch/S$, we write $(\lSch/S)_t$ for the site defined by $t$.
If $\cG$ is a sheaf of abelian groups on $S_t$ and $X \in \lSch_{qcqs}/S,$ 
we write $R_t\Gamma(X, \cG)$ for the right derived functor of $\Gamma(X_t, -)$ evaluated on $\cG.$ 
The assignment $X \mapsto R_t\Gamma(X, \cG)$ defines a $t$-sheaf on $\lSch_{qcqs}/S$
with values in the derived category of abelian groups.

Let $\cC$ be an $\infty$-category with small limits and $\cF \colon \lSch_{qcqs}^{\text{op}}/S\to \cC$ 
be a functor.
A morphism $f \colon Y \to X$ in $\lSch_{qcqs}/S$ satisfies $\cF$-\emph{descent} 
if the naturally induced morphism 
$$\cF(X) \to \lim \left( \cF(Y) \rightrightarrows \cF(Y \times_X Y) \rightrightrightarrows \cdots \right)$$
is an equivalence.
We say that $f$ satisfies \emph{universal} $\cF$-descent if all 
base changes of $f$ in $\lSch_{qcqs}/S$ satisfy $\cF$-descent.

\begin{definition}\label{def:logvsheaf}
Let $S$ be a qcqs fs log scheme and $\cC$ an $\infty$-category with small limits. 
A functor $\cF \colon \lSch_{qcqs}^{\text{op}}/S \to \cC$ 
satisfies \emph{log $v$-descent} if the following conditions hold:
\begin{enumerate}
\item[(i)] $\cF$ takes finite disjoint unions of fs log schemes to products.
\item[(ii)] Every log $v$-cover $f \colon Y \to X$ in $\lSch_{qcqs}/S$ satisfies universal $\cF$-descent.
\end{enumerate}
\end{definition}

Next we describe three basic facts for morphisms of universal $\cF$-descent taken from 
\cite[Lemma 3.4]{MR4278670} that we will use repeatedly.

\begin{lemma}\label{lem:descentproperties}
Let $\cF \colon \lSch_{qcqs}^{\text{op}}/S \to \cC$ be a functor, 
and $f \colon Y \to X$, $g \colon Z \to Y$ be morphisms in $\lSch_{qcqs}/S$.
\begin{enumerate}[(i), ref =\cref{lem:descentproperties}.(\roman*)]
\item\label{lem:descentproperties.(i)} If $f$ has a section $s \colon X \to Y$, then $f$ satisfies universal $\cF$-descent.
\item\label{lem:descentproperties.(ii)} If $f$ and $g$ satisfy universal $\cF$-descent, 
then the composite $f \circ g \colon Z \to X$ satisfies universal $\cF$-descent. 
\item\label{lem:descentproperties.(iii)} If the composite $f \circ g \colon Z \to X$ satisfies universal $\cF$-descent, then $f$ satisfies universal $\cF$-descent.
\end{enumerate}
\end{lemma}

\subsection{Log étale cohomology}
\label{subsection:Kummer étale cohomology with torsion coefficients}

Our aim is to show that log étale cohomology with torsion coefficients satisfies decent for 
log $v$-covers. 
This implies representability of log \'etale cohomology groups 
in the $\infty$-category of $S^1$-stable logarithmic 
$h$-motives $\ul{\logSH}_{\lh}^\eff(S),$ see \cref{subsection:Logarithmic H-Motives}.

We begin with a few lemmas needed in the proof of \cref{thm:etalecohomology}.

\begin{lemma}
\label{lem:Kummer}
Let $P\to Q$ be a Kummer homomorphism of saturated monoids,
and consider its \v{C}ech nerve $Q^\bullet$ in the category of saturated monoids.
Then there is an isomorphism
\[
Q^n
\cong
Q\oplus (Q^\gp/P^\gp)^{\oplus n}
\]
for every integer $n\geq 0$.
In particular,
there is an isomorphism $\ol{Q^n}\cong \ol{Q}$.
\end{lemma}
\begin{proof}
The claim is evident when $n=0$.
The case $n=1$ is \cite[Lemma 3.3]{MR1922832}.
An induction argument establishes the claim for $n\geq 2$.
\end{proof}

\begin{lemma}\label{lem:toric2}
Let $k$ be a separably closed field and $\theta\colon P\to Q$ be a homomorphism of 
sharp fs monoids such that $\theta^\gp \colon P^{\gp} \to Q^{\gp}$ is an isomorphism.
Then for $n>1$ there is a natural quasi-isomorphism
$$R\Gamma_{\ket}(\pt_{Q,k}, \Z/n)\simeq R\Gamma_{\ket}(\pt_{P,k}, \Z/n)$$
\end{lemma}
\begin{proof}
The profinite group $I(\pt_{P,k})$ and hence the category 
$\text{$I(\pt_{P,k})$-$\Z/n$-Mod}/_{\ul{\pt_{P,k}}}$ in \cite[Proposition 4.6]{MR1457738} 
depends only on $P^\gp$ and $k$.
Consequently,
the equivalence 
\[
\Sh_\ket(\pt_{Q,k},\Z/n)
\simeq
\Sh_\ket(\pt_{P,k},\Z/n)
\]
implies our claim.
\end{proof}

\begin{lemma}
\label{lem:cohomology of log point}
For an fs log scheme $X$ and $n>1$ there is a natural quasi-isomorphism
\[
R\Gamma_\ket(X\times (\P^1,0+\infty),\Z/n)
\simeq
R\Gamma_\ket(X\times \pt_\N,\Z/n)
\]
\end{lemma}
\begin{proof}
Let $f\colon X\times (\P^1,0+\infty)\to X$ and $g\colon X\times \pt_\N \to X$ be the projections.
It suffices to show that the induced morphism of complexes $f_*f^*\Z/n\to g_*g^*\Z/n$ is a quasi-isomorphism.
Due to the proper base change theorem \cite[Theorem 5.1]{MR1457738},
we reduce to the case $X=\pt:=\Spec{\Z}$.

Consider the strict closed immersions 
$i' \colon \pt_{\N} \to (\P^1, 0+\infty)$ and $i \colon \pt \to \square$ at the point $0 \in \P^1$, 
together with their open complements $j' \colon (\P^1-\{0\}, \infty) \to (\P^1, 0+\infty)$ and $j \colon  (\P^1-\{0\}, \infty) \to \square.$ 
Here, the projective line $\P^1$ is over the base $\Spec{\Z}$.
Let $p \colon (\P^1, 0+\infty) \to \square$ be the morphism that forgets the log structure at $0 \in \P^1.$ Then these morphisms fit together into the commutative diagram
\begin{equation}
\label{eq:tikzlogschemes}
\begin{tikzcd}
\pt_\N \ar[r, "i'"] \ar[d, "q"'] & (\P^1,0+\infty) \ar[d, "p"] & \ar[l, "j'" above] (\P^1-\{0\},\infty) \ar[d,"\id"]
\\
\pt \ar[r, "i"] & \square & \ar[l, "j" above] (\P^1-\{0\},\infty)
\end{tikzcd}
\end{equation}
By specializing the localization sequence 
\cite[§2.8(4)]{MR1457738} to the top row of \eqref{eq:tikzlogschemes} 
we obtain
\begin{equation}\label{eq:localiationsequence}
j'_!j'^* \to \id \to i'_*i'^*
\end{equation}
and the fiber sequence of complexes
\begin{equation}
\label{eq:toprowlocalizationsequence}
R\Gamma_{\ket}((\P^1,0+\infty),j'_!j'^*\Z/n)
\to
R\Gamma_{\ket}((\P^1,0+\infty),\Z/n)
\to
R\Gamma_{\ket}((\P^1,0+\infty),i'_*i'^*\Z/n)
\end{equation}
Similarly, the localization sequence 
\begin{equation}
j_!j^* \to \id \to i_*i^*
\end{equation}
gives the fiber sequence of complexes
\begin{equation}\label{eq:bottomrowlocalizationsequence}
R\Gamma_{\ket}(\square,j_!j^*\Z/n)
\to
R\Gamma_{\ket}(\square,\Z/n)
\to
R\Gamma_{\ket}(\square,i_*i^*\Z/n)
\end{equation}
We may identify $R\Gamma_{\ket}(\square,i_*i^*\Z/n)$ with $R\Gamma_{\ket}(\pt,\Z/n),$ and 
since Kummer étale cohomology is $\square$-invariant by \cite[Theorem 9.1.5]{logSH}, we also have that
$$R\Gamma_{\ket}(\square,\Z/n) \simeq R\Gamma_{\ket}(\pt,\Z/n)$$
By \eqref{eq:bottomrowlocalizationsequence} this implies the quasi-isomorphism
\begin{equation}\label{eq:localizationzero}
R\Gamma_{\ket}(\square,j_!j^*\Z/n) \simeq 0
\end{equation}
Moreover, since $p_*j'_! \simeq j_!$ by \cite[§5.4]{MR1457738}, we obtain 
$$R\Gamma_{\ket}((\P^1,0+\infty),j'_!j'^*\Z/n)\simeq 
R\Gamma_{\ket}(\square,j_!j^*\Z/n)$$
which by \eqref{eq:localizationzero} shows that
$$R\Gamma_{\ket}((\P^1,0+\infty),j'_!j'^*\Z/n) \simeq 0$$
Using \eqref{eq:toprowlocalizationsequence} we deduce the quasi-isomorphism
\begin{equation}\label{eq:projectiveequalspointN}
R\Gamma_{\ket}((\P^1,0+\infty),\Z/n) \simeq R\Gamma_{\ket}((\P^1,0+\infty),i'_*i'^*\Z/n)
\end{equation}
However, $R\Gamma_{\ket}((\P^1,0+\infty),i'_*i'^*\Z/n)$ can be identified with $R\Gamma_{\ket}(\pt_\N,\Z/n)$, which by \eqref{eq:projectiveequalspointN} implies the claim.
\end{proof}

\begin{lemma}
\label{lem:cohomology}
Let $k$ be a separably closed field of characteristic $p$ and $P$ be a sharp fs monoid.
If $n>1$ is coprime to $p$,
there are natural quasi-isomorphisms
\[
R\Gamma_\ket(\pt_{P,k},\Z/n)
\simeq
R\Gamma_\ket(\A_{P,k},\Z/n)
\simeq
R\Gamma_\ket(\A_{P^\gp,k},\Z/n)
\]
\end{lemma}
\begin{proof}
We consider $P$ as a cone $\sigma$ in the lattice $P^\gp$.
Let $\Sigma$ be a fan with the single maximal cone $\sigma$.
By toric resolution of singularities \cite[Theorem 11.1.9]{toric},
there exists a subdivision of fans $\Delta\to \Sigma$ such that $\Delta$ is smooth.
Choose any maximal cone $\delta$ of $\Delta$,
and regard $\delta$ as an fs monoid $Q$,
which is isomorphic to $\N^d$ for some integer $d\geq 0$..
The inclusion $Q\to P$ induces an isomorphism $Q^\gp\xrightarrow{\cong}P^\gp$.
With these definitions there is a naturally induced commutative diagram
\[
\begin{tikzcd}
R\Gamma_\ket(\pt_{Q,k},\Z/n)\ar[d]\ar[r,leftarrow]&
R\Gamma_\ket(\A_{Q,k},\Z/n)\ar[d]\ar[r]&
R\Gamma_\ket(\A_{Q^\gp,k},\Z/n)\ar[d]
\\
R\Gamma_\ket(\pt_{P,k},\Z/n)\ar[r,leftarrow]&
R\Gamma_\ket(\A_{P,k},\Z/n)\ar[r]&
R\Gamma_\ket(\A_{P^\gp,k},\Z/n)
\end{tikzcd}
\]
The right horizontal morphisms are quasi-isomorphisms by \cite[Theorem 7.4]{MR1922832},
the right vertical morphism is a quasi-isomorphism since $Q^\gp\cong P^\gp$,
and the left vertical morphism is a quasi-isomorphism by \cref{lem:toric2}.
Hence it suffices to show the upper left horizontal morphism is a quasi-isomorphism.
By \cref{lem:cohomology of log point} there is a quasi-isomorphism
\[
R\Gamma_\ket(\pt_{Q,k},\Z/n)
\simeq
R\Gamma_\ket(\Spec(k)\times (\P^1,0+\infty)^d,\Z/n)
\]
To conclude, 
we appeal to the quasi-isomorphism 
(due to \cite[Theorem 7.4]{MR1922832})
\[
R\Gamma_\ket(\Spec(k)\times (\P^1,0+\infty)^d,\Z/n)
\simeq
R\Gamma_\ket(\A_{Q,k},\Z/n)
\]
\end{proof}

\begin{lemma}
\label{lem:no log point}
Let $k$ be a separably closed field of characteristic $p>0$ and $P$ be a sharp fs monoid.
Then there is a natural quasi-isomorphism
\[
R\Gamma_\ket(k,\Z/p)
\simeq
R\Gamma_\ket(\pt_{k,P},\Z/p)
\]
\end{lemma}
\begin{proof}
By the Artin-Schreier exact sequence of strict \'etale sheaves on $\lSch/k$
\[
0
\to
\Z/p
\to
\cO
\to
\cO
\to
0
\]
it suffices to show there is a naturally induced quasi-isomorphism
\[
R\Gamma_\ket(k,\cO)
\to
R\Gamma_\ket(\pt_{k,P},\cO)
\]
This is a map from $R\Gamma_\ket(k,\cO)\simeq R\Gamma_\et(k,\cO)\simeq \Gamma(k,\cO)$ to 
$R\Gamma_\ket(\pt_{k,P},\cO)\simeq \Gamma(\pt_{k,P},\cO)$ (see, e.g., \cite[Corollary 3.26]{MR2452875}),
which is the identity map on $k$.
\end{proof}

For a log scheme $X$,
we set $X_\mathrm{red}:=X\times_{\ul{X}}\ul{X}_\mathrm{red}$.

\begin{lemma}\label{lem:reduced}
Let $k$ be a separably closed field,
$X\to \pt_{P.k}$ be a strict morphism of fs log schemes.
For $n>1$ there is a quasi-isomorphism
$$R\Gamma_{\ket}(X, \Z/n)\simeq R\Gamma_{\ket}(X_{\textnormal{red}}, \Z/n)$$
\end{lemma}
\begin{proof}
Proposition 4.6 in \cite{MR1457738} reduces our claim to the étale sites of the 
underlying schemes, which is \cite[Théorème IV.18.1.2]{EGA}.
\end{proof}

For an fs log scheme $X$ and a geometric point $\ol{x}$ of $\ul{X}$,
the \emph{strict henselization of $X$ at $\ol{x}$} is the fs log scheme strict over $X$ 
such that the underlying scheme is the strict henselization of $\ul{X}$ at $\ol{x}$.

Let $\cC$ be an $\infty$-category with small limits and filtered colimits.
A functor $\lSch^{\text{op}}_{qcqs}/S \to \cC$ is called \emph{finitary} 
if it preserves small filtered colimits with strict affine transition morphisms.
Let $\cG$ be a sheaf on $S_\ket.$
The functor sending an fs log scheme $X \in \lSch/S$ to the bounded below chain complex 
$R\Gamma_{\ket}(X, f^*\cG)$ is finitary, see \cite[Lemma 4.2]{MR1457738},
where $f\colon X\to S$ is the structure morphism.
The following is a logarithmic version of \cite[Lemma 5.1]{MR4278670}, 
whose proof is essentially the same.

\begin{lemma}\label{lem:5.1}
Let $\cF \colon \lSch_{qcqs}^{\text{op}}/S \to D(\Z)^{\geq 0}$ be a finitary functor. 
Suppose $\cF$ satisfies strict étale descent. 
If the base change of the morphism $f\colon Y \to X$ in $\lSch_{qcqs}/S$ along 
any strict henselizations of $X$ satisfy $\cF$-descent, 
then $f$ satisfies $\cF$-descent.
\end{lemma}
\begin{proof}
Our claim states there is a naturally induced quasi-isomorphism 
\begin{equation}\label{eq:quasi-isomorphicsheaves}
\cF(X) \to \lim \left( \cF(Y) \rightrightarrows \cF(Y \times_X Y) \rightrightrightarrows\cdots \right)
\end{equation}
Since $\cF$ satisfies strict étale descent, we view 
$X'\mapsto \cF(X')$ and 
$$X' \mapsto \lim \left( \cF(X'\times_X Y) \rightrightarrows \cF(X'\times_X Y \times_X Y)
\rightrightrightarrows
\cdots \right)$$ 
as strict \'etale sheaves.
The former is finitary by assumption and the latter is finitary since the totalization in 
$D(\Z)^{\geq 0}$ commutes with filtered colimits.
For the latter, 
note that totalization of cosimplicial objects in $D(\mathbb{Z})^{\ge 0}$ commutes with filtered colimits. 
Indeed, by the $\infty$-categorical Dold--Kan correspondence, 
the degree-$n$ term of the total complex is a finite limit of terms in bidegrees $(p,q)$ with $p,q\ge 0$ and $p+q=n$. Since filtered colimits of abelian groups are exact, 
they commute with these finite limits. 
Hence the latter is also finitary.
Our assumption says that these two sheaves are isomorphic on stalks after base change to 
strict henselizations of $X$. 
Since $\cF$ satisfies strict étale descent 
and $X_{\set}\simeq \ul{X}_\et$ has enough points by \cite[Tag 03PU]{stacks-project},
thus \eqref{eq:quasi-isomorphicsheaves} is a quasi-isomorphism.
\end{proof}

The next result is a generalization of \cite[Exposé Vbis]{SGA4} and \cite[Proposition 5.3.3]{CDEtale} 
(see also \cite[Proposition 5.2]{MR4278670}) to logarithmic geometry.

\begin{theorem}[Log $v$-descent for log étale cohomology with torsion coefficients] 
\label{thm:etalecohomology}
Suppose $S$ is a scheme with trivial log structure and $\cG$ is a torsion sheaf on 
$S_{\letale}\simeq S_{\et}$.
Then the functor $\cF \colon (\lSch_{qcqs}/S)^{\textnormal{op}} \to D(\Z)^{\geq 0}$ given by
\[
(\alpha \colon X \to S)
\mapsto
R\Gamma_{\letale}(X,\alpha^*\cG)
\]
satisfies log $v$-descent.
\end{theorem}

We begin with two lemmas for the functor $\cF$ in \cref{thm:etalecohomology}.

\begin{lemma}
\label{lem:reduction}
Let $f\colon Y\to X$ be a morphism in $\lSch_{qcqs}/S$.
If the pullback $Y\times_X Z\to Z$ satisfies $\cF$-descent for every morphism $Z\to X$ in 
$\lSch_{qcqs}/S$ such that $\ul{Z}$ is the spectrum of a separably closed field,
then $f$ satisfies universal $\cF$-descent.
\end{lemma}
\begin{proof}
It suffices to show that $f$ satisfies $\cF$-descent. 
By \cref{lem:5.1} we may assume $\ul{X}$ is a strictly henselian 
local ring with a separably closed residue field $k$.
The proper base change theorem \cite[Theorem 5.1]{MR1457738} yields 
$\cF(Y)\simeq \cF(Y\times_{\ul{X}} \Spec(k))$, 
and similarly for the other terms in the \v{C}ech nerve of $Y\to X$. 
To conclude,
observe that the underlying scheme of $X\times_{\ul{X}}\Spec{k}$ is $\Spec{k}$.
\end{proof}

\begin{lemma}
\label{lem:AP}
Let $P$ be an fs monoid and assume $S$ is the spectrum of a field of characteristic $p$.
For $m\geq 2$,
the morphism $\A_m\colon S\times \A_P\to S\times \A_P$ induced by the multiplication by 
$m$ homomorphism $m\colon P\to P$ satisfies universal $\cF$-descent.
\end{lemma}
\begin{proof}
Using the prime factorization of $m$ and \ref{lem:descentproperties.(ii)},
we may assume $m$ is a prime number.
If $m\neq p$,
then $\A_m$ is a Kummer \'etale cover,
so $\A_m$ satisfies universal $\cF$-descent.

Let $g\colon X\to S\times \A_P$ in $\lSch_{qcqs}/S$ be a morphism in $\lSch_{qcqs}/S$ 
such that $\ul{X}=\Spec(k)$ is the spectrum of a separably closed field.
When $m=p>0$, 
by \cref{lem:reduction}
it suffices to show the 
pullback $f\colon Y\to X$ of $\A_m$ along $g$ satisfies $\cF$-descent.
Let $P\to Q$ be a chart of $X\to S\times \A_P$,
and consider $Q':=Q\oplus_{P,p}P$.
Since $\A_m$ is the relative Frobenius, 
it is a universal Kummer homeomorphism in the sense of \cite[D\'efinition 2.1]{zbMATH01594939} 
by \cite[Th\'eor\`eme 2.4]{zbMATH01594939}.
Thus $\ul{Y}_\red \cong \Spec{k}$.

Similarly,
$\ul{Y_i}_\red \cong \Spec{k}$ for every term $Y_i$ in the \v{C}ech nerve of $Y\to X$.
We use \cref{lem:Kummer} to identify 
\[
Y_\red \rightrightarrows (Y\times_X Y)_\red
\rightrightrightarrows
(Y\times_X Y\times_X Y)_\red 
\rightrightrightrightarrows
\cdots
\]
with the constant diagram for $M':=\ol{Q'}$
\[
\pt_{M',k} \rightrightarrows \pt_{M',k}
\rightrightrightarrows
\pt_{M',k}
\rightrightrightrightarrows
\cdots
\]
By \cref{lem:toric2}
it suffices to show $\cF(\pt_{M,k})\simeq \cF(\pt_{M',k})$,
where $M:=\ol{Q}$.

For an fs monoid $N$ we form the profinite group
\[
I(N):=\lim_{\gcd(n,p)=1} \Hom(N^\gp,\ker(k^*\xrightarrow{n} k^*))
\]
By \cite[Proposition I.1.3.5.(2)]{Ogu}
$M^\gp$ and $M'^\gp$ are torsion free.
Let $\{x_1,\ldots,x_d\}$ be a basis of the free abelian group $M^\gp$.
Then for each $i$, $\Z x_i\oplus_{P^\gp,p}P^\gp$ is isomorphic to $\Z y_i$ for some $y_i\in M'^\gp$.
Moreover, the homomorphism $\Z x_i\to \Z y_i$ is the identity or multiplication by $p$.
Using the basis $\{y_1,\ldots,y_d\}$ of $M'^\gp$,
the homomorphism $M^\gp\to M'^\gp$ becomes a diagonal matrix with entries $1$ or $p$,
so that the naturally induced homomorphism $I(M)\to I(M')$ is an isomorphism.
Proposition 4.6 in \cite{MR1457738} implies the equivalence $\cF(\pt_{M,k})\simeq \cF(\pt_{M',k})$.
\end{proof}

\begin{proof}[Proof of \cref{thm:etalecohomology}]
We must verify the two sheaf conditions in \cref{def:logvsheaf}.
As a consequence of log \'etale descent,
the functor $\cF$ takes disjoint unions of fs log schemes to products. 
It remains to show that every log $v$-cover $f \colon Y \to X$ 
of fs log schemes satisfies universal $\cF$-descent.
Log \'etale descent allows us to work dividing Zariski locally on $X$.
Hence, 
by \cite[Theorem 1.1]{integrallogblowup}, 
we may assume $f$ is integral.
We may also assume that $X$ has a chart $P$ for some fs monoid $P$.
Since $\cG$ is a torsion sheaf, \cite[Proposition 5.4 (2)]{MR3658728} yields an equivalence
\[
R\Gamma_\ket(X,\alpha^*\cG)
\simeq
R\Gamma_\letale(X,\alpha^*\cG)
\]
By \cref{lem:reduction} we may assume that $S$ is the spectrum of a separably closed field 
of characteristic $p$.
We claim there exists an integer $m\geq 1$ such that the pullback 
$Y\times_{\A_P,\A_m} \A_P\to X\times_{\A_P,\A_m} \A_P$ of $f$ is saturated,
where $\A_m\colon \A_P\to \A_P$ is induced by the multiplication map $m\colon P\to P$.
If $f$ has a chart, this follows from \cite[Theorem I.4.9.1]{Ogu}. 
In general,
let $\{Y_i\to Y\}_{i\in I}$ be a (finite) Zariski covering such that each $Y_i\to X$ admits a chart; by the above, there is a 
corresponding integer $m_i$.
Let $m$ be the product of all the $m_i$.
Then the pullback $Y_i\times_{\A_P,\A_m} \A_P\to X\times_{\A_P,\A_m} \A_P$ of $f$ is saturated for each $i$ and hence $Y\times_{\A_P,\A_m} \A_P\to X\times_{\A_P,\A_m} \A_P$ is saturated too.
Owing to \cref{lem:AP} we may assume $f$ is saturated.

Since $\cF$ is finitary, 
arguing as in \cite[Proof of Proposition 5.2]{MR4278670} we may assume $\ul{f}$ is finitely presented; hence it admits a refinement that is the composite of a
quasi-compact open covering and a proper surjective morphism of schemes by \cite[Theorem 3.12]{Rydh}.
Consider the factorization $Y\xrightarrow{g} \ul{Y}\times_{\ul{X}}X\xrightarrow{h} X$ of $f$.
Observe that $\ul{g}$ is an isomorphism and $h$ is strict.
Since $\ul{h}=\ul{f}$, $h$ admits a refinement that is a composite of quasi-compact open covering and a strict proper surjective morphism.
By \ref{lem:descentproperties.(ii)}
we reduce to proving universal $\cF$-descent for the following three cases: 
$f$ is a quasi-compact open covering, 
$f$ is strict proper surjective, 
and $\ul{f}$ is an isomorphism.

Kummer \'etale descent settles the case of a quasi-compact open covering
since every quasi-compact open covering is an example of a Kummer \'etale covering.

For strict proper surjective morphisms, 
\cref{lem:5.1} allows us to assume $\ul{X}=\pt_{P,k}$ for some fs monoid $P$ 
and separably closed field $k$.
Let $\ol{k}$ be an algebraic closure of $k$.
Since $f$ is strict proper surjective,
there exists a strict closed immersion $\pt_{P,\ol{k}}\to Y$.
By \ref{lem:descentproperties.(iii)}
we reduce to the case $Y:=\pt_{P,\ol{k}}$.
The diagram
\[
Y_\red
\rightrightarrows
(Y\times_X Y)_\red
\rightrightrightarrows
(Y\times_X Y\times_X Y)_\red
\rightrightrightrightarrows
\cdots
\]
can be identified with the constant diagram
\[
Y
\rightrightarrows
Y
\rightrightrightarrows
Y
\rightrightrightrightarrows
\cdots
\]
Hence, by \cref{lem:reduced},
it suffices to show $\cF(X)\simeq \cF(Y)$.
To conclude, note that, 
as in the proof of \cref{lem:reduced}, 
there is an equivalence of categories of Kummer \'etale sheaves of $\Z$-modules
\[
\Sh_\ket((\pt_{P,k})_{\ket},\Z)
\simeq
\Sh_\ket((\pt_{P,\ol{k}})_{\ket},\Z)
\]

Next, we discuss the case in which \(\ul{f}\) is an isomorphism. Since we assume that \(f\) is saturated, 
the terms in the \v{C}ech nerve of \(f\) share the same underlying scheme, \(\ul{X}\). 
By \cref{lem:reduction}, it suffices to demonstrate that the pullback of \(f\) along a morphism \(Z \to X\) in \(\lSch_{qcqs}/S\), where \(\ul{Z}\) is the spectrum of a separably closed field, satisfies \(\cF\)-descent. 
By replacing \(X\) with \(Z\), we can assume that \(\ul{X}\) is the spectrum of a separably closed field, and we only need to show that \(f\) satisfies \(\cF\)-descent rather than universal \(\cF\)-descent. 
Note that \(X \cong \pt_{P,k}\) for some sharp fs monoid \(P\), as stated in \cref{exm:log point}. Since \(\ul{f}\) is an isomorphism, we have \(\ul{Y} \cong \Spec(k)\), which implies that \(Y \cong \pt_{Q,k}\) for some sharp fs monoid \(Q\), according to \cref{exm:log point}. 
The induced homomorphism \(\Gamma(X,\cM_X) \to \Gamma(Y,\cM_Y)\) takes the form \(\theta: P \oplus k^* \to Q \oplus k^*\). By restriction, we obtain the homomorphism \(\ol{\theta}: P \to Q\).

Furthermore, the morphism \( X \to S \) factors through \( \ul{X} \cong \Spec(k) \), which allows us to assume that \( S = \Spec(k) \). Consequently, the sheaf \( \mathcal{G} \) on \( S_\text{ket} \) becomes a constant sheaf associated with a torsion abelian group \( G \). Since every torsion abelian group is a filtered colimit of finitely generated abelian groups, and because Kummer étale cohomology commutes with filtered colimits \cite[Lemma 5.3]{MR3658728}, we may further assume \( G = \mathbb{Z}/\ell \) for some prime number \( \ell \).

Each term $Y_i$ in the \v{C}ech nerve of $f$ has underlying scheme $\Spec{k}$ because $\ul{f}$ is an 
isomorphism and $f$ is saturated.
The description of the fiber product in the category of log schemes in \cite[Proposition III.2.1.2]{Ogu} 
yields an isomorphism
\[
\Gamma(Y_i,\cM_Y)
\cong
(Q\oplus k^*)\oplus_{P\oplus k^*} \cdots \oplus_{P\oplus k^*} (Q\oplus k^*)
\]
Together with \cite[Proposition I.4.2.5.(5)]{Ogu} we have
\[
\ol{\Gamma(Y_i,\cM_Y)}
\cong
Q\oplus_P \cdots \oplus_P Q =: Q^i
\]
Thus $Y_i\cong \pt_{Q^i,k}$.
The \v{C}ech nerve of $f$ is
\[
\pt_{Q^0,k} \rightrightarrows \pt_{Q^1,k}
\rightrightrightarrows
\pt_{Q^2,k}
\rightrightrightarrows
\cdots
\]

When $\ell=p$, we use \cref{lem:no log point} to identify 
\begin{equation}
\label{eqn:etalecohomology}
\cF(\pt_{Q^0,k}) \rightrightarrows \cF(\pt_{Q^1,k})
\rightrightrightarrows
\cF(\pt_{Q^2,k})
\rightrightrightrightarrows
\cdots
\end{equation}
with the constant diagram
\[
\cF(k) \rightrightarrows \cF(k)
\rightrightrightarrows
\cF(k)
\rightrightrightrightarrows
\cdots
\]
Thus $Y\to X$ satisfies $\cF$-descent.

When $\ell\neq p$, 
we use \cref{lem:cohomology} to identify \eqref{eqn:etalecohomology}
with 
\[
\cF(Y') \rightrightarrows \cF(Y'\times_{X'}Y')
\rightrightrightarrows
\cF(Y'\times_{X'}Y'\times_{X'} Y') 
\rightrightrightrightarrows
\cdots
\]
where $X':=\A_{P^\gp,k}$, $Y':=\A_{Q^\gp,k}$, and $f'\colon Y'\to X'$ is induced by $\theta$.
Since $\A_{\Z,k}\cong \G_{m,k}$, $f'$ has a section.
This concludes the proof by \ref{lem:descentproperties.(i)}.
\end{proof}

\begin{remark}
\label{remark:notstraigthforward}
\cref{thm:etalecohomology} has a 
scheme-theoretic precursor 
\cite[Proposition 5.2]{MR4278670}. 
However, 
the proof does not directly generalize to the logarithmic setting.
If $Y$ is a proper scheme over an algebraically closed field $k$, 
then the structure morphism $Y\to \Spec{k}$ admits a section.
This is used in \cite[Proof of Proposition 5.2]{MR4278670}.
However, if $Y$ is a proper fs log scheme over $k$,
then $Y\to \Spec{k}$ does not necessarily admit a section, 
e.g., when $Y=\pt_{\N,k}$.
\end{remark}

\subsection{Logarithmic differentials in characteristic 0}
\label{subsection:Logarithmic differentials in characteristic 0}
Working in characteristic zero, 
we study the log $h$-sheafification of logarithmic differentials.
The interesting fact that differential forms mesh well with the $h$-topology is 
independently due to Huber--Jörder \cite{HuberJorder}, \cite{Huber}, 
and Lee \cite{1251832562}; 
more precisely, 
the $h$-sheafification $\Omega_\h^q$ defines differential 
forms on singular schemes \cite[Theorem 3.6, Corollary 6.5]{HuberJorder}. 
Extending a cohomology theory from smooth schemes to singular schemes using proper 
$h$-covers goes back to Deligne's work on mixed Hodge structures \cite{HodgeIII}. 

We begin by recalling the definition of the sheaf of logarithmic differentials, 
see \cite[Chapter IV]{Ogu} for a full treatment. 
Let $f \colon X \to S$ be a morphism of fs log schemes. 
We write $\Omega^1_{X/S}$ for the sheaf of $\cO_X$-modules generated by universal log derivations 
\cite[Theorem IV.1.2.4]{Ogu}. 
By \cite[Theorem IV.1.2.4]{Ogu} it is the sheaf of $\mathcal{O}_X$-modules given by
\begin{equation}
\label{eqn:log differential}
\Omega^1_{X/S} = \big( \Omega_{\ul{X}/\ul{S}}^1 \oplus (\mathcal{O}_X \otimes \mathcal{M}^{\gp}_X) \big) / \sim
\end{equation}
Here, 
$\sim$ is the sub $\mathcal{O}_X$-module generated by sections of the form
\begin{enumerate}
 \item[(i)] $(d \alpha(a), 0) - (0, \alpha(a) \otimes a)$ for $a \in \mathcal{M}_X,$ and
 \item[(ii)] $(0, 1 \otimes a)$ for $a \in \text{Im}(f^{-1} \mathcal{M}_S \to \mathcal{M}_X)$,
\end{enumerate}
where $\alpha\colon \cM_X\to \cO_X$ is the log structure morphism.
For $q\geq 0$, the \emph{sheaf of log differentials} on $X$ is defined as 
\[
\Omega^q_{X/S} := 
\left\{
\begin{matrix*}[l]
    \cO_X & q = 0 \\
    \wedge^q \Omega^1_{X/S} & q > 0
\end{matrix*}
\right.
\]

We denote the sheaf of differentials as $\Omega^q_X$ when $S = k$, where $k$ is a field. If $X$ is a fine and saturated (fs) log scheme of finite type over $k$, then $\Omega^q_X$ is a coherent log étale sheaf, as stated in \cite[Chapter 9.1, Lemma 9.1.1]{logDM}. Additionally, if $X$ belongs to the category of smooth schemes over $k$ (denoted as $\SmlSm/k$), then $\Omega^q_X$ is a locally free $\mathcal{O}_X$-module, according to \cite[Proposition IV.3.2.1]{Ogu}.
For $q \geq 0$, let $\Omega^q$ be the presheaf on the category of log schemes over $k$ (denoted as $\lSch/k$), defined by 
$$X \longmapsto \Gamma(X, \Omega^q_X)$$

\begin{example}\label{exm:smoothdifferentials}
When $X = (\ul{X}, \partial X)$ is a smooth log smooth log scheme over a field $k$, where the boundary $\partial X$ (étale locally) is given in local coordinates 
$$V(x_1, x_2 \hdots, x_r) \subset \A^n$$
the sheaf of log differentials $\Omega^i_X$ agrees with the sheaf of differentials $\Omega^i_{\ul{X}}(\log \partial X)$ with log poles along $\partial X.$
Hence $\Omega^i_X$ is generated by the symbols
\[
d\log x_1, \hdots, d\log x_r, dx_{r+1}, \hdots, dx_n
\]
For instance, if $\partial X$ is the strict normal crossing variety 
$$\Spec{\big( k[x_1, x_2, \hdots, x_n]/(x_1 x_2 \cdots x_r) \big)}$$
then $\Omega^1_X$ is generated by symbols
\[d\log x_1, \hdots, d\log x_r, dx_{r+1}, \hdots, dx_n\]
\end{example}

\begin{proposition}\label{prop:dividingdescent}
Every dividing cover $Y \to X$ of noetherian fs log schemes of finite Krull dimension over a field $k$ satisfies universal $R\Gamma(-,\Omega^q)$-descent for all $q \geq 0.$
In particular, there is a naturally induced quasi-isomorphism
\[
R\Gamma(X,\Omega^q)
\simeq
R\Gamma(Y,\Omega^q)
\]
\end{proposition}
\begin{proof}
We refer to \cite[Corollary 6.17]{BLPO}.
\end{proof}

\begin{definition}
For $q \geq 0,$ 
the sheaf of \textit{log $h$-differential forms} $\Omega^q_{\lh}$ is the log $h$-sheafification of the presheaf $\Omega^q$ on $\lSch/k.$
\end{definition}

We conjecture that $R\Gamma_{\lh}(-,\Omega^q_{\lh})$ is homotopy invariant in the following sense.

\begin{conjecture}\label{conj:loghdifferantialshomotopyinvariant}
Let $k$ be a field of characteristic 0.
The log $h$-sheaf $R\Gamma_{\lh}(-,\Omega^q_{\lh})$ is $(\P^m, \P^{m-1})$-invariant for all $m\geq 1$ and $q\geq 0$.
That is, 
for all $X \in \lSch/k$, 
the projection $X \times (\P^m, \P^{m-1}) \to X$ induces a quasi-isomorphism
$$R\Gamma_{\lh}(X,\Omega^q_{\lh}) \longrightarrow R\Gamma_{\lh}(X \times (\P^m, \P^{m-1}),\Omega^q_{\lh})$$
\end{conjecture}

\begin{conjecture}\label{conj:loghdifferentialsequal}
Let $k$ be a field of characteristic 0, and
let $X$ be a smooth scheme with trivial log structure. Then the morphism
$$R\Gamma(X, \Omega^q) \longrightarrow R\Gamma_{\lh}(X, \Omega^q_{\lh})$$
is a quasi-isomorphism for all $q\geq 0.$
\end{conjecture}

\begin{remark}
We refer to \cite[Theorem 4.7]{Geisser} and \cite[Corollary 6.5]{HuberJorder} 
for similar statements for schemes in the $eh$- and $h$-topology, respectively.
A technical obstruction for generalizing \cite[Theorem 4.7]{Geisser} to our setting 
is the lack of a logarithmic version of \cite[Corollary 2.6]{Geisser}.
For example,
the log point
$\pt_{\N,k}$
does not admit a log $h$-cover that is log smooth over $k$.
\end{remark}

Conjecturally, we obtain the following logarithmic analog of \cite[Corollary 6.5]{HuberJorder}.
\begin{theorem}\label{thm:loghdifferentials}
Let $k$ be a field of characteristic 0 and $X \in \lSm/k$.
\cref{conj:loghdifferantialshomotopyinvariant,conj:loghdifferentialsequal} imply 
there is a quasi-isomorphism
$$R\Gamma(X,\Omega^q) \overset{\simeq}{\longrightarrow} R\Gamma_{\lh}(X,\Omega^q_{\lh})$$
for all $q\geq 0.$
\end{theorem}
\begin{proof}
By \cref{prop:dividingdescent} we reduce the question to $X \in \SmlSm/k$.
We perform induction on the number $d$ of irreducible components of the strict normal crossing divisor $\partial X$ of $X.$ 
If $d=0,$ see \cref{conj:loghdifferentialsequal}. 
Suppose the claim holds for $d-1$, and let $Z$ be an irreducible component of $\partial X$.
Consider $Y\in \SmlSm/k$ such that $\ul{Y}=\ul{X}$ and $\partial Y$ is obtained by removing 
$Z$ from $\partial X$.
By \cref{conj:loghdifferantialshomotopyinvariant} the log $h$-sheaf 
$R\Gamma_{\lh}(-,\Omega^q_{\lh})$ is representable in $\logSH^\eff(k)$.
Applying \cite[Theorem 7.5.4]{logDM} to both $R\Gamma(-,\Omega^q)$ and 
$R\Gamma_{\lh}(-,\Omega^q_{\lh})$ we obtain a morphism of fiber sequences
$$
\begin{tikzcd}
R\Gamma(\mathrm{Th}(\rN_Z Y), \Omega^q) \ar[r]\ar[d] &  R\Gamma(Y, \Omega^q) \ar[r] \ar[d, "\simeq"]  & R\Gamma(X, \Omega^q)\ar[d]
\\
R\Gamma_{\lh}(\mathrm{Th}(\rN_Z Y), \Omega^q_{\lh}) \ar[r]&  R\Gamma_{\lh}(Y, \Omega^q_{\lh}) \ar[r] & R\Gamma_{\lh}(X, \Omega^q_{\lh})
\end{tikzcd}
$$
where the middle morphism is a quasi-isomorphism by the induction hypothesis.
By \cite[Proposition 7.4.5]{logDM} there is a natural quasi-isomorphism
$$
R\Gamma(\mathrm{Th}(\rN_ZY),\Omega^q)
\simeq
\fib\big(R\Gamma(\P(\rN_ZY\oplus \cO),\Omega^q)\to R\Gamma(\P(\rN_ZY),\Omega^q) \big)
$$
and similarly for $R\Gamma_\lh(\mathrm{Th}(\rN_ZY), \Omega_\lh^q)$.
The boundaries of $\P(\rN_ZY\oplus \cO)$ and $\P(\rN_ZY)$ have $d-1$ irreducible components 
since each irreducible component is a projective bundle over the corresponding irreducible 
component of $Z$.
Hence the induction hypothesis ensures there is a quasi-isomorphism
$$R\Gamma(\mathrm{Th}(\rN_ZY), \Omega^q) 
\overset{\simeq}{\longrightarrow} R\Gamma_\lh(\mathrm{Th}(\rN_Z Y), \Omega_\lh^q)$$
This concludes the proof.
\end{proof} 

\begin{corollary}\label{cor:loghsheafificationisaloghsheaf}
Let $k$ be a field of characteristic $0$.
\cref{conj:loghdifferantialshomotopyinvariant,conj:loghdifferentialsequal} 
imply that $R\Gamma(-,\Omega^q)$ is a sheaf on $(\lSm/k)_{\lh}$.
\end{corollary}
\begin{proof}
\cref{prop:dividingdescent} reduces the question to $\SmlSm/k$. 
\cref{thm:loghdifferentials} finishes the proof.
\end{proof}

\subsection{Stable logarithmic \texorpdfstring{$h$}{h}-motives}
\label{subsection:Logarithmic H-Motives}
In this section we show that the constant torsion sheaf $\Z/n$ is representable in the category of $S^1$-stable log $h$-motives.

\begin{definition}
Let $S$ be a qcqs scheme and let $t$ be a topology on $\lSch/S.$
Let $\Sh_t(\lSch_{qcqs}/S,\Sp)$ denote the $\infty$-category of $t$-sheaves on $\lSch_{qcqs}/S$ with values in spectra. By \cite[Proposition 3.2.15]{logSH} the cyclic permutation 
$$(\P^1, \infty)^{\wedge 3} \overset{\sigma_{123}}{\longrightarrow} (\P^1, \infty)^{\wedge 3}$$
is homotopic to the identity, so by \cite[Theorem 2.26]{zbMATH06374152} one can consider the $\infty$-category that formally inverts (all powers of) the object $(\P^1, \infty),$ 
denoted 
$$\underline{\logSH}_t^\eff(S) := \Sh_t(\lSch_{qcqs}/S,\Sp)[(\P^\bullet,\P^{\bullet-1})^{-1}]$$
We define the symmetric monoidal $\infty$-category\footnote{We follow \cite{CD09} and write $\underline{\logSH}_t(S)$ 
to emphasize that it is the stable motivic homotopy category of (fs log) schemes $\lSch_{qcqs}/S$ instead of 
(log) smooth (log) schemes $\lSm/S$.}
\begin{gather*}
\underline{\logSH}_t(S):=
\Sp_{\P^1}(\Sh_t(\lSch_{qcqs}/S,\Sp)[(\P^\bullet,\P^{\bullet-1})^{-1}])
\end{gather*}
as the $\infty$-category of $\P^1$-spectra in this $\infty$-category.
\end{definition}

\begin{remark}\label{rmk:loghtopologynofibreproducts}
The log $h$-topology is not preserved under fiber products of log smooth fs log schemes: 
The morphism $\square \to \P^1$ is a log $h$-cover, 
but the fiber product $\square\times_{\P^1}\square$ is not log smooth, 
hence not an object of $\SmlSm/k.$ 
If we try to rectify this by insisting that
$\square\times_{\P^1}\square = \square$, 
then we have an isomorphism of motives $M(\square)\cong M(\P^1)$ and $\P^1$ becomes contractible.
However, this is not desirable in the logarithmic setting.
For this reason, 
we employ qcqs fs log schemes in the definition of stable logarithmic $h$-motives.
\end{remark}

An fs log scheme $X\in \lSch_{qcqs}/S$ represents a presheaf of spaces on $\lSch_{qcqs}/S$ 
defined by $Y \mapsto \Hom_{S}(Y, X).$ 
Consider its $t$-sheafification, 
and let 
$$\Sigma^\infty_{S^1,+} \colon \lSch_{qcqs}/S \to \underline{\logSH}_{t}^\eff(S)$$ 
denote the functor sending an fs log scheme $X$ to its pointed infinite $S^1$-suspension.
A sheaf $\cF$ on $(\lSch_{qcqs}/S)_t$ is called \emph{representable} in 
$\underline{\logSH}_{t}^\eff(S)$ 
if there exist an object $M \in \underline{\logSH}_{t}^\eff(S)$ such that
$$H^{p}_{t}(X, \cF) \cong \pi_0\mathrm{Map}(\Sigma^\infty_{S^1,+} X, \Sigma^p_{S^1} M)$$
for all $p \in \Z,$ where $\mathrm{Map}(-,-)$ denotes the mapping space.

\begin{proposition}\label{prop:Z/nrepresentable}
For $n>1$,
$R\Gamma_\letale(-,\Z/n)$ is representable in the $\infty$-category of $S^1$-stable log $h$-motives
$\ul{\logSH}_{\lh}^\eff(S)$.
\end{proposition}
\begin{proof}
By \cref{thm:etalecohomology} it remains to show 
$R\Gamma_\letale(-,\Z/n)$ satisfies $(\P^m, \P^{m-1})$-invariance for $m\geq 1$.
For $X\in \lSch_{qcqs}/S$,
consider the presheaf on $\SmlSm/\ul{X}$ defined by 
\begin{equation}\label{eq:smoothpresheaf}
Y \longmapsto R\Gamma_{\letale}(Y \times_{\ul{X}} X, \Z/n)
\end{equation}
The presheaf \eqref{eq:smoothpresheaf} satisfies dividing descent by \cref{prop:dividingdescent} 
and $\square$-invariance according to \cite[Lemma 6.9.4]{MR1922832}. 
Since \eqref{eq:smoothpresheaf} is defined on $\SmlSm/\ul{X}$, 
appealing to both 
\cite[Proposition 7.3.1]{logDM} and \cite[Proposition 3.2.18]{logSH} 
(to remove the noetherian hypothesis), 
it also satisfies $(\P^m, \P^{m-1})$-invariance,
which concludes the proof.
\end{proof}

Let $k$ be a field of characteristic $p>0$.
Consider the $\infty$-category
$$\SH_{ h}^\eff(k)
:=
\Sh_\h(\Sm/k,\Sp)[(\A^1)^{-1}].
$$
Then
\[
R\Gamma_{\et}(-,\Z/p)\notin \SH_{h}^\eff(k)
\]
since $R\Gamma_{\et}(\A_k^1,\Z/p)\not\simeq R\Gamma_{\et}(k,\Z/p)$ even though $R\Gamma_\et(-,\Z/p)\in \Sh_\h(\Sm/k,\Sp)$ by
\cite[Proposition 5.2]{MR4278670}.
Hence $R\Gamma_{\ket}(-,\Z/p)$ makes a clear distinction between $\SH_{h}^\eff(k)$ and 
$\ul{\logSH}_{\lh}^\eff(k)$.
\vspace{0.1in}

We end this section with a few observations.
Voevodsky's first construction of triangulated categories of motives over a base scheme $S$ 
\cite[\S 4.1]{Voevodsky-Homology} used the $h$- and $qfh$-topologies; 
these are $\Z[1/p]$-linear categories due to \cite[Proposition 4.1.7]{Voevodsky-Homology}.
Building on the work in this paper, 
it is natural to study a logarithmic analog of Voevodsky's $qfh$-topology. 
Moreover, 
one can ask whether the $\infty$-category
\[
\ul{\mathrm{logDA}}_{\lh}^\eff(S):=\Sh_{\lh}(\lSch_{qcqs}/S,\rD(\mathrm{\Z}))
[(\P^\bullet,\P^{\bullet-1})^{-1}]
\]
and its log $qfh$-version satisfies the projective bundle formula 
\cite[Theorem 4.2.7]{Voevodsky-Homology} when $S$ is equipped with a nontrivial log structure.
The notion of presheaves with log transfers in \cite[Definition 4.1.1]{logDM} is currently only 
available over a field $k$; in fact, 
many of the arguments in \cite{logDM} break down over log points.
Resolving these questions could lead to 
``triangulated categories of log $h$- and $qfh$-motives over log points'' 
without involving log transfers.

\bibliography{biblogSH}
\bibliographystyle{siam}

\end{document}